\documentclass[11pt]{article}

\usepackage[utf8]{inputenc}
\usepackage[a4paper,top=2.54cm,bottom=2.5cm,left=2.54cm,right=2.54cm]{geometry}
\usepackage{graphicx}
\usepackage{booktabs}
\usepackage{array}
\usepackage{paralist}
\usepackage{verbatim} 
\usepackage{subfig}
\usepackage[usenames,svgnames,dvipsnames]{xcolor}
\usepackage{amsmath}
\usepackage{amssymb}
\usepackage{amsthm}
\usepackage{amsbsy}
\usepackage{babel}
\usepackage[pdftex,plainpages=false,colorlinks=true,citecolor=DarkGreen,linkcolor=Blue,urlcolor=black,filecolor=black,bookmarksopen=true]{hyperref}
\usepackage{setspace}
\usepackage{lineno}
\usepackage{enumitem}
\usepackage{mathrsfs}
\usepackage{tikz}
\usepackage{tikz-cd}
\usetikzlibrary{calc,fadings,decorations.pathreplacing}
\usepackage{pgfplots}
\pgfplotsset{compat=newest}
\usepgfplotslibrary{dateplot}
\usepackage{pgf}
\usetikzlibrary{arrows}
\usetikzlibrary{matrix}
\usetikzlibrary{pgfplots.groupplots}

\hypersetup{pdftitle={On pointwise error estimates for Voronoï-based finite volume methods for the Poisson equation on the sphere},
 pdfauthor={Leonardo A. Poveda and Pedro Peixoto}}
 
\numberwithin{equation}{section}

\usepackage{fancyhdr}
\usepackage[titles,subfigure]{tocloft} % Alter the style of the Table of Contents

\newtheorem{thm}{Theorem}[section]
\newtheorem{prp}[thm]{Proposition}
\newtheorem{lem}[thm]{Lemma}

\theoremstyle{definition}
\newtheorem{defn}[thm]{Definition}

\theoremstyle{remark}
\newtheorem{rem}[thm]{Remark}
\numberwithin{equation}{section}

\def\eg{\textit{e.g.}}
\def\ie{\textit{i.e.}}
\def\io{\"{o}}
\def\Or{{\mathcal O}}
\def\R{\mathbb{R}}
\def\Rp{\mathbb{R}_{+}}

\def\Rt{\mathbb{R}^{3}}

\def\S{\mathbb{S}^{2}}
\def\Sh{\mathbf{S}_{h}}
\def\Oh{\Omega_{h}}
\def\Oijk{\Omega_{ijk}}
\def\d{\mathrm{d}}
\def\hi{h_{i}}
\def\x{\mathrm{x}}
\def\y{\mathrm{y}}
\def\z{\mathrm{z}}
\def\m{\mathrm{m}}
\def\q{\mathrm{q}}
\def\xij{\x_{ij}}
\def\mij{\m_{ij}}
\def\xn{\x_{n}}
\def\xp{\x^{\ast}}
\def\yp{\y^{\ast}}

\def\xpij{\xp_{ij}}
\def\xpjk{\xp_{jk}}
\def\xpik{\xp_{ik}}
\def\xip{\x_{i}}
\def\xjp{\x_{j}}
\def\xkp{\x_{k}}

\def\xci{\xip^{c}}

\def\SN{\mathrm{S}_{N}}
\def\a{m_{a}}
\def\l{m_{l}}
\def\Vi{\widetilde{\mathrm{V}}_{i}}
\def\Vj{\widetilde{\mathrm{V}}_{j}}

\def\Vn{\widetilde{\mathrm{V}}_{n}}
\def\Tgi{{\mathrm T}_{ijk}}
\def\Tgsi{\widetilde{\mathrm T}_{ijk}}
\def\Qi{\widetilde{\mathrm Q}_{i}}
\def\Qj{\widetilde{\mathrm Q}_{j}}
\def\Qk{\widetilde{\mathrm Q}_{k}}
\def\Qn{\widetilde{\mathrm Q}_{n}}
\def\Sn{\mathrm{S}_{n}}
\def\Si{\mathrm{S}_{i}}
\def\Sj{\mathrm{S}_{j}}
\def\Sk{\mathrm{S}_{k}}
\def\Qip{\mathrm{Q}_{i}}
\def\Qjp{\mathrm{Q}_{j}}
\def\Qkp{\mathrm{Q}_{k}}
\def\Qnp{\mathrm{Q}_{n}}
\def\Prj{{\mathcal P}}
\def\Prji{\Prj^{-1}}
\def\Th{\mathcal{T}_{h}}
\def\Ths{\widetilde{\Th}}
\def\VGs{\widetilde{\mathcal{V}}_{h}}
\def\LiN{\Lambda(i)} 
\def\gij{\widetilde{\Gamma}_{ij}}
\def\tij{\widetilde{\tau}_{ij}}

\def\cl{\mathrm{cl}}
\def\al{\alpha}
\def\pr{\partial}

\def\pri{\partial_{i}}
\def\prj{\partial_{j}}

\def\prsi{\partial_{s,i}}
\def\grad{\nabla}
\def\grads{\nabla_{s}}

\def\laps{\Delta_{s}}
\def\lapd{\mathrm{L}_{s,h}}
\def\nsx{\vec{\mathrm{n}}_{\S,\x}}
\def\nxgi{\vec{\mathrm{n}}_{\x,\gij}}
\def\nxvi{\vec{\mathrm{n}}_{\x,\Vi}}
\def\nxvn{\vec{\mathrm{n}}_{\x,\Vn}}
\def\nxti{\vec{\mathrm{n}}_{\x,\Tgsi}}
\def\nyti{\vec{\mathrm{n}}_{\y,\Tgsi}}
\def\nxqn{\vec{\mathrm{n}}_{\x,\Qn}}
\def\TxS{\mathbb{T}_{\x,\S}}
\def\TyS{\mathbb{T}_{\y,\S}}
\def\Ct{C^{2}(\S)}
\def\Lt{L^{2}(\S)}
\def\Li{L^{\infty}(\S)}
\def\Lo{L^{1}(\S)}
\def\Lp{L^{p}(\S)}
\def\LpT{L^{p}(\Tgsi)}
\def\Lq{L^{q}(\S)}
\def\LqT{L^{q}(\Tgsi)}

\def\Ho{H^{1}(\S)}
\def\Ht{H^{2}(\S)}
\def\WtpT{W^{2,p}(\Tgsi)}
\def\HqS{H_{0}^{1}(\S)}
\def\HtHq{H^{2}(\S)\cap\HqS}
\def\Wkp{W^{k,p}(\S)}

\def\Wti{W^{2,\infty}(\S)}
\def\WtiHq{W^{2,\infty}(\S)\cap\HqS}

\def\Wtp{W^{2,p}(\S)}

\def\WmpT{W^{m,p}(\Tgsi)}
\def\WoqT{W^{1,q}(\Tgsi)}
\def\WtpHq{\Wtp\cap\HqS}

\def\Po{{\mathbb P}_{1}(\Sh,\Th)}
\def\Pos{{\mathbb P}_{1}(\S,\Ths)}
\def\U{\widetilde{U}_{h}}
\def\Vhs{\widetilde{V}_{h}}
\def\A{\mathcal{A}}
\def\Fcij{\widetilde{\mathcal{F}}_{ij}}
\def\Fdij{\overline{\mathcal{F}}_{ij}}
\def\Fdji{\overline{\mathcal{F}}_{ji}}
\def\Ahs{\widetilde{\A}_{h}}
\def\Ahd{\overline{\A}_{h}}

\def\uh{u_{h}}
\def\vh{v_{h}}
\def\wh{w_{h}}
\def\vah{\varphi_{h}}
\def\Eu{u^{\Omega}}

\def\Eup{\overline{u}^{\Omega}}
\def\uhp{\overline{u}_{h}}
\def\Euhp{\overline{u}^{\Omega}_{h}}

\def\Evh{\vh^{\Omega}}

\def\IU{\widetilde{\Pi}_{h}}
\def\IV{\widetilde{\Pi}^{\ast}_{h}}
\def\I{\widetilde{\mathrm{I}}_{h}}
\def\eh{\varepsilon_{h}}
\def\PV{\mathrm{P}_{h}}
\def\Gy{G^{\y}}

\def\Gyx{\Gy(\x)}

\def\Gxz{G^{\x}(\z)}
\def\dy{\delta^{\y}}
\def\dyx{\dy(\x)}
\def\dyz{\dy(\z)}

\def\Edy{\boldsymbol{\delta}^{\y,\Omega}}
\def\Edyp{\delta^{\yp,\Omega}}

\def\Gyd{g^{\y}}

\def\Gydo{\Gyd_{0}}
\def\Gyda{\Gyd_{1}}
\def\Gydn{\Gyd_{n}}
\def\Gydoh{\Gyd_{0,h}}

\def\Gydah{\Gyd_{1,h}}
\def\Gydnh{\Gyd_{n,h}}
\def\Gydx{\Gyd(\x)}

\def\xaf{x^{\ast}_{1}}
\def\xas{x^{\ast}_{2}}
\def\xat{x^{\ast}_{3}}
\def\bxat{\overline{x}^{\ast}_{3}}
\def\br{\vec{\boldsymbol{r}}}
\def\bth{\vec{\boldsymbol{\theta}}}
\def\bva{\vec{\boldsymbol{\phi}}}
\def\prt{\pr^{2}}

\providecommand{\keywords}[1]{\small\textbf{Keywords:} #1}

\title{\bf On pointwise error estimates for Voronoï-based finite volume methods for the Poisson equation on the sphere}
\author{Leonardo A. Poveda\thanks{Department of Mathematics, The Chinese University of Hong Kong, Shatin, Hong Kong SAR. Corresponding author (lpoveda@math.cuhk.edu.hk)} \and Pedro Peixoto\thanks{Instituto de Matemática e Estatística, Universidade de São Paulo, Brazil}}
\begin{document}
\maketitle
%\tableofcontents
%\linenumbers
\begin{abstract}
In this paper, we give pointwise estimates of a Voronoï-based finite volume approximation of the Laplace-Beltrami operator on Voronoï-Delaunay decompositions of the sphere. These estimates are the basis for a local error analysis, in the maximum norm, of the approximate solution of the Poisson equation and its gradient. Here, we consider the Voronoï-based finite volume method as a perturbation of the finite element method. Finally, using regularized Green's functions, we derive quasi-optimal convergence order in the maximum-norm with minimal regularity requirements. Numerical examples show that the convergence is at least as good as predicted.
\end{abstract}


\section{Introduction}
\label{sec:introduction}

The study of numerical solutions of Partial Differential Equations (PDEs) posed on spheres and other surfaces arise naturally in many applications. It is especially important in geophysical fluid dynamics, where a sphere is usually adopted as a domain. For instance, in weather forecasting and climate modeling, PDEs are used and discretized through finite differences, finite elements,  and finite volume methods on a spherical domain \cite{giraldo1997lagrange,stuhne1999new,heikes1995anumerical}.

The efficiency and accuracy of the approximate solutions depend on certain characteristics of the discretization of the sphere and of the differential operators. In the present work, we start by looking at the approximations of solutions of the Poisson problem on the unit sphere using \emph{spherical icosahedral geodesic grids} \cite{sadourny1968integration,williamson1968integration,baumgardner1985icosahedral,heikes1995anumerical}, \ie, grids that come from an icosahedron inscribed on the sphere, with its vertices and faces projected onto the spherical surface. We obtain a spherical triangular grid whose edges are geodesic arcs and satisfy the so-called Delaunay criterion \ie, maximizing the smallest angle \cite{glitzky2010discrete,gartner2019why,hjelle2006triangulations}. Following  \cite{augenbaum1985construction,renka1997algorithm}, the mentioned construction will allow us to define two types of grids on $\S$: a triangular Delaunay (primal) decomposition and a Voronoï (dual) decomposition. In what follows, we will consider such grids in a more general way, not necessarily restricted to those directly built from an icosahedron. Therefore, we will refer to them as general Voronoï-Delaunay decompositions of the sphere throughout this paper.

Voronoï-based finite volume methods are very popular and allow great flexibility. However, the finite volume scheme may not be formally consistent, yet can still lead to second-order convergence results, leading to what is sometimes called supra-convergence \cite{barbeiro2009supraconvergent,despres2004lax,bouche2005error,pascal2007supraconvergence,manteuffel1986numerical,kreiss1986supra,diskin2010notes}. Error estimates of the first order of convergence for approximate solutions of planar Voronoï-based finite volume method in the $H^{1}$ and $L^{2}$-norm have been reported by   \cite{mishev1998finite,gallouet2000error,eymard2001finite,du2003voronoi,eymard2006cell}. Second-order accuracy in the $L^{2}$-norm using planar dual Donald decompositions, which uses the triangle barycenters as vertices of the dual grid, were considered in \cite{jianguo1998finite,li2000generalized,chou2000error,chou2007superconvergence} and for general surfaces in \cite{ju2009finite,ju2009posteriori}. These latter works explicitly use properties of the barycentric dual cells, \ie, the quadratic order is obtained by using the centroid of triangles, and therefore these constructions cannot be generally extended to the arbitrary Voronoï-Delaunay decompositions.

Previous work \cite{li2000generalized,chou2000error,chen2002note,wu2003error} also usually have an extra requirement in the regularity of the exact solution (belong to $H^{3}$ or $W^{3,p}$), which is excessive compared to the requirements in finite element methods \cite{brenner2007mathematical,ciarlet2002finite,rannacher1982some}. However, \cite{ewing2002accuracy} has reported a sufficient condition to decrease the regularity of the exact solution (is in $H^{2}$) but imposes an added regularity requirement on the forcing source term, which is in $H^{1}$ to get the optimal convergence order. The authors also highlighted that, except for one-dimensional domains or the solution domain has a boundary smooth enough, the $H^{1}$-regularity of the source term does not automatically imply the $H^{3}$-regularity of the exact solution.

On the sphere, \cite{du2005finite} show a quadratic order estimate in the Voronoï-Delaunay decomposition on $\S$ in the $L^{2}$-norm and a Spherical Centroidal Voronoï Tesselation (SCVT) optimized grid  \cite{du1999centroidal,du2003constrained} is used as their Voronoï-Delaunay decomposition and the excessive regularity assumption in the exact solution. Based on this, as a first minor result of this work, we show more general error estimates than given in \cite{du2005finite} applying the approach of \cite{ewing2002accuracy}, \ie, decreasing the regularity in the exact solution and imposing a minimum regularity in the source term. Thus we obtain the desired order of convergence. We conclude, as in \cite{du2005finite}, that the proof cannot be extended to Voronoï-Delaunay decompositions in general due to the explicit use of the criteria of the SCVT. In general, determining a quadratic convergence order in the $L^{2}$-norm is an open issue for the Voronoï-based finite volume method. Several efforts have been made to answer this question and are limited to topological aspects; for example, \cite{omnes2011second} gives a partial answer and an important improvement in relation to what is known so far.

The topic of this paper is the convergence analysis for approximate solutions of the Voronoï-based finite volume method in the maximum-norm, extending existing results for the plane \cite{ewing2002accuracy,zhang2014superconvergence,chou2003lp} to the sphere. We do not know that advances in this direction have been previously described in the literature, particularly for Voronoï-Delaunay decompositions on $\S$. We opted to consider the Voronoï-based finite volume method as a perturbation of the finite element method, an approach widely described in the literature \cite{li2000generalized, ewing2002accuracy,lin2013finite}. 

The main idea is to use the standard error estimation procedures developed for the finite elements on surfaces such as \cite{demlow2009higher, kroner2017approximative,kovacs2018maximum} along with the use of the regularized Green's functions on the sphere.

The main result of our work is the proof of the sub-linear convergence order of a classic finite volume method on general spherical Voronoï-Delaunay decompositions, for the Poisson equation, in the maximum norm. This result tightens the gap between theoretical convergence analysis and existing empirical evidence for the convergence of such schemes in the sphere. Empirical evidence indicated the possibility of linear convergence; however, here, our results contain a logarithmic factor caused by the use of linear functions in the primal Delaunay decomposition, which apparently cannot be avoided, as also was initially examined in Euclidean domains 
via finite element methods by  \cite{scott1976optimal,rannacher1982some,schatz1998pointwise} and more recently described in \cite{lyekekhman2017maximum,leykekhman2021weak,li2022maximum}. Additionally, linear convergence order in the maximum norm is proved for SCVT grids.

The outline of the paper is as follows: in Section \ref{sec:problsetting}, we briefly introduce some notation and the model equation used in this work. Section \ref{sec:grids} is devoted to the usual recursive construction of the spherical icosahedral geodesic grids. In Section \ref{sec:fvm}, we establish the classical finite volume method and discrete function spaces. The error estimates for the discrete finite volume scheme are given in Section \ref{sec:erroranalysis}. Section \ref{sec:numericalexps} gives numerical experiments and final comments.

\section{Problem setting}
\label{sec:problsetting}
In this section, we start defining the model problem, some notations, and function spaces that will be used throughout the paper. Let $\S:=\{\x\in\Rt:\|\x\|=1\}$ be the unit sphere, where $\|\cdot\|$ represents the Euclidean norm in $\Rt$. Let $\grads$ denote the tangential gradient \cite{dziuk2013finite} on $\S$ defined by,
\[
\grads u(\x)=\grad u(\x)-\left(\grad u(\x)\cdot\nsx\right)\nsx,
\]
where $\grad$ denotes the usual gradient in $\Rt$ and $\nsx$ represents the unit outer normal vector to $\S$ at $\x=(x_{1},x_{2},x_{3})$. We shall adopt standard notation for Sobolev spaces on $\S$ (see \eg,  \cite{hebey1996sobolev}). Given $1\leq p\leq \infty$ and $k$ non-negative integer, we denote the Sobolev spaces by,
\[
\begin{split}
L^{p}(\S)&=\left\{u(\x):\int_{\S}|u(\x)|^{p}ds(\x)<\infty\right\},\\
\Wkp&=\left\{u\in L^{p}(\S):\grads^{\alpha}u\in L^{p}(\S),\quad\mbox{for }0\leq|\al|\leq k\right\},
\end{split}
\]
where $\grads^{\al}=\grad_{s,1}^{\al_{1}}\grad_{s,2}^{\al_{2}}\grad_{s,3}^{\al_{3}}$ is the multi-index notation for the weak tangential derivatives up to order $k$, where $\al=(\al_{1},\al_{2},\al_{3})$ is a vector of non-negative integers with $|\al|=\al_{1}+\al_{2}+\al_{3}$. The function space $\Wkp$ is equipped with the norm
\[
\|u\|_{W^{k,p}(\S)}=
\begin{cases}
\left(\sum_{0\leq|\al|\leq k}\|\grads^{\al}u\|_{L^{p}(\S)}^{p}\right)^{1/p},&\mbox{for }1\leq p< \infty\\
\max_{0\leq |\al|\leq k}\|\grads^{\al}u\|_{L^{\infty}(\S)},& \mbox{for }p=\infty.
\end{cases}
\]
We set $H^{k}(\S)=W^{k,2}(\S)$ along with the standard inner product
\[
(u,v)=\int_{\S}u(\x)v(\x)ds(\x),\quad\mbox{for all }u,v\in L^{2}(\S),
\]
where $ds(\x)$ is the surface area measure. Additionally, we define the zero-averaged subspace of $\Ho$ as
\[
\HqS:=\left\{u\in\Ho:\int_{\S}u(\x)ds(\x)=0\right\},
\]
equipped with the $H^{1}$-norm. Throughout this paper, we use $C_{\ast}$ as a generic positive real constant, which may vary with the context and depends on the problem data and other model parameters. Also, we will use the notation $a\preceq b$, which means that there is a positive constant $C$, independent of the mesh size, such that $a\leq Cb$.

We now introduce the Poisson equation to be considered. Let $f\in\Lt$ be a given forcing (source) satisfying the compatibility condition,
\begin{equation}
\label{eq:compf}
\int_{\S}f(\x)ds(\x)=0.
\end{equation}
The model problem consists of finding a scalar function $u:\S\to \R$ satisfying 
\begin{equation}
\label{eq:strongform}
-\laps u(\x)=f(\x),\quad\mbox{for each }\x\in\S,
\end{equation}
where $-\laps=-\grads\cdot\grads$ denotes the Laplacian  on $\S$. We impose $\int_{\S}u(\x)ds(\x)=0$ to ensure uniqueness of solution. For any $u,v\in\HqS$, we define the bilinear functional $\A:\HqS\times\HqS\to\R$ such that
\begin{equation}
\label{eq:biformFE}
\A(u,v)=\int_{\S}\grads u(\x)\cdot \grads v(\x)ds(\x).
\end{equation}
The bilinear functional is well-defined on the space $\HqS\times\HqS$ and is continuous and coercive, i.e.,
there are positive constants $C_0$ and $C_1$ such that 
\[
\begin{split}
\left|\A(u,v)\right|&\leq C_{0}\|\grads u\|_{\Lt}\|\grads v\|_{\Lt},\quad\mbox{for each }u,v\in\HqS,\\
\A(u,u)&\geq C_{1}\|\grads u\|_{\Lt}^{2},\quad\mbox{for each }u\in\HqS.
\end{split}
\]
The variational formulation of \eqref{eq:strongform} reads: find $u\in\HqS$ such that
\begin{equation}
\label{eq:weakform}
\A(u,v)=(f,v),\quad\mbox{for each }v\in\HqS,
\end{equation}
in which $(f,v)=\int_{\S}f(\x)v(\x)ds(\x)$ and the source data $f$ satisfies \eqref{eq:compf}.

As a consequence of the Lax-Milgram theorem \cite{ciarlet2002finite,brenner2007mathematical}, problem \eqref{eq:weakform} has a unique solution. This solution is such that for some positive constant $C_{S}$,
\begin{equation}
\label{eq:stability}
\|\grads u\|_{\Lt}\leq C_{S}\|f\|_{\Lt}.
\end{equation}
Moreover, we have the regularity property: for $f\in\Lt$ satisfying \eqref{eq:compf}, the unique weak solution $u\in\HtHq$ of \eqref{eq:weakform} satisfies, for some positive constant $C_{R}$, 
\begin{equation}
\label{eq:regularity}
\|u\|_{\Ht}\leq C_{R}\|f\|_{\Lt}.
\end{equation}
A detailed proof of \eqref{eq:regularity} can be found in \cite[Theorem 3.3 pp.~304]{dziuk2013finite}.

\section{Spherical icosahedral geodesic grids}
\label{sec:grids}
In this section, we describe the discretization framework to approximate the sphere $\S$, following  \cite{baumgardner1985icosahedral,giraldo1997lagrange,heikes1995anumerical}. The spherical icosahedral grid can be constructed by defining an icosahedron inscribed inside $\S$, which has triangular faces and vertices. Each edge of the original icosahedron whose vertices are on $\S$ is projected onto the surface of $\S$.  We employ a recursive refinement of the grid by connecting the midpoints of the geodesic arcs to generate four sub-triangles in each geodesic triangle. This procedure may be applied to all geodesic triangles of the initial icosahedron to create a grid of desired resolution (see Figure \ref{fig:refinamenttriangle}).

\begin{figure}[!h]
\centering
\includegraphics[scale=0.26]{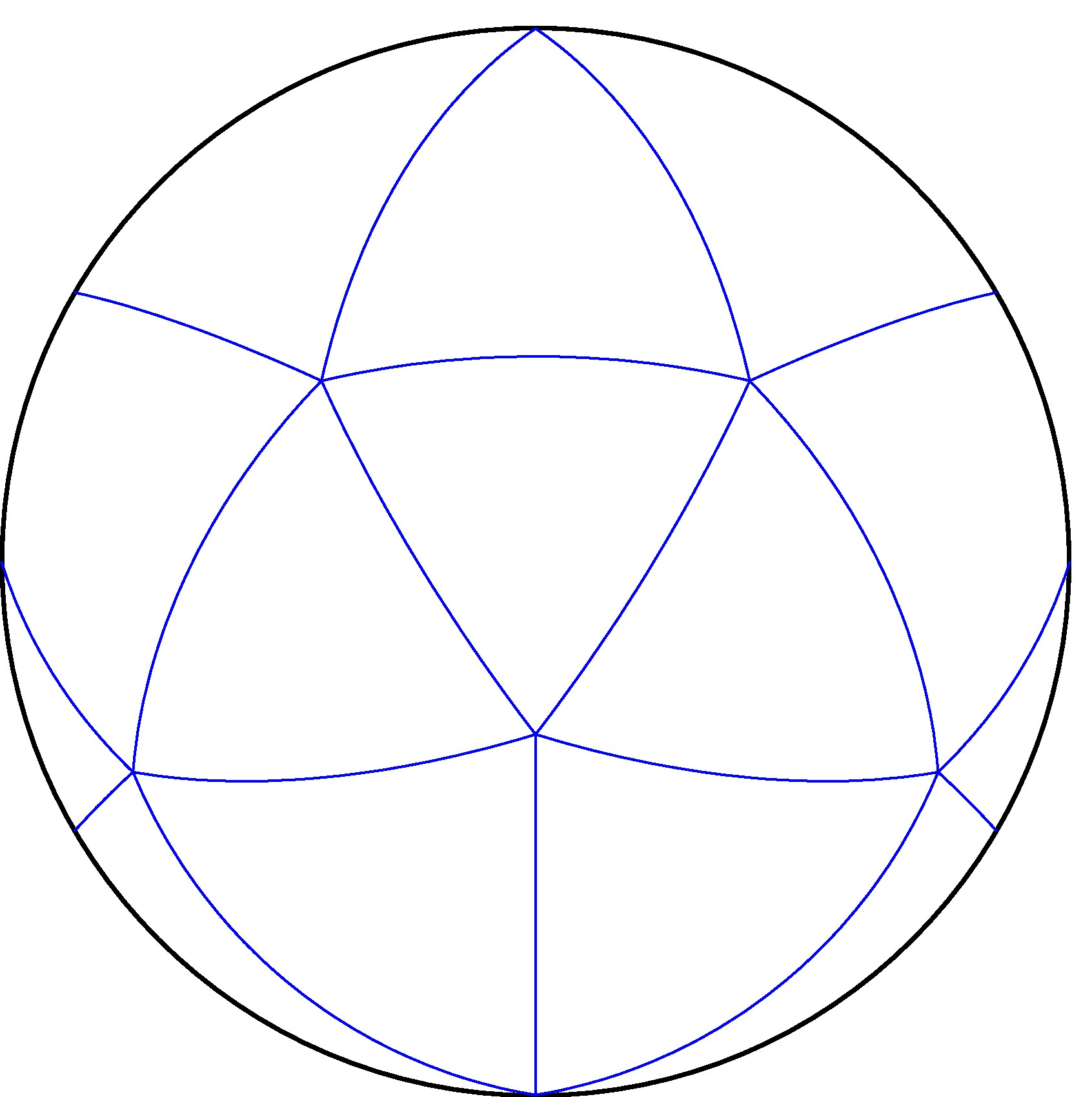}\hspace{0.2cm}
\includegraphics[scale=0.26]{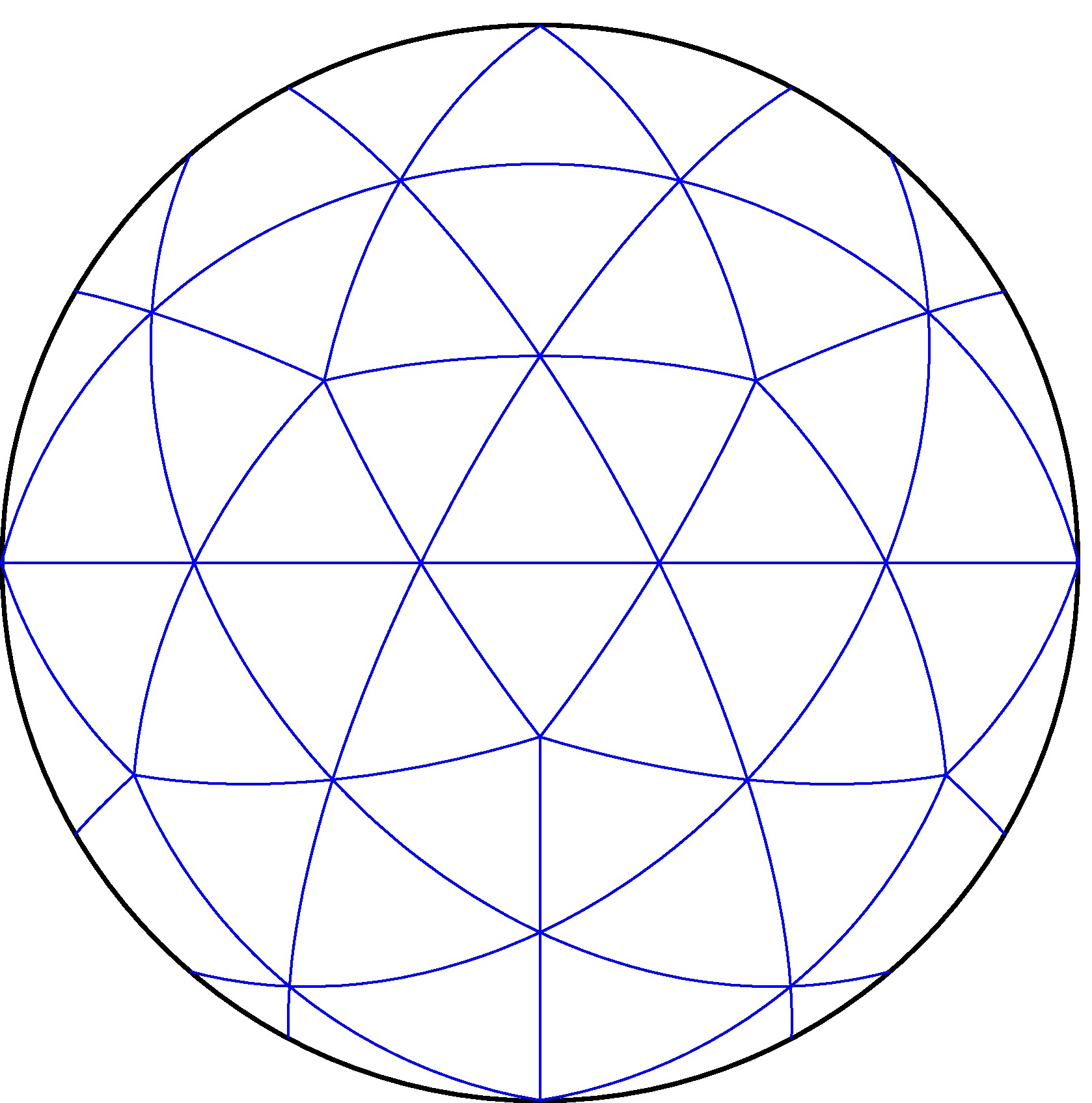}
\hspace{0.2cm}
\includegraphics[scale=0.26]{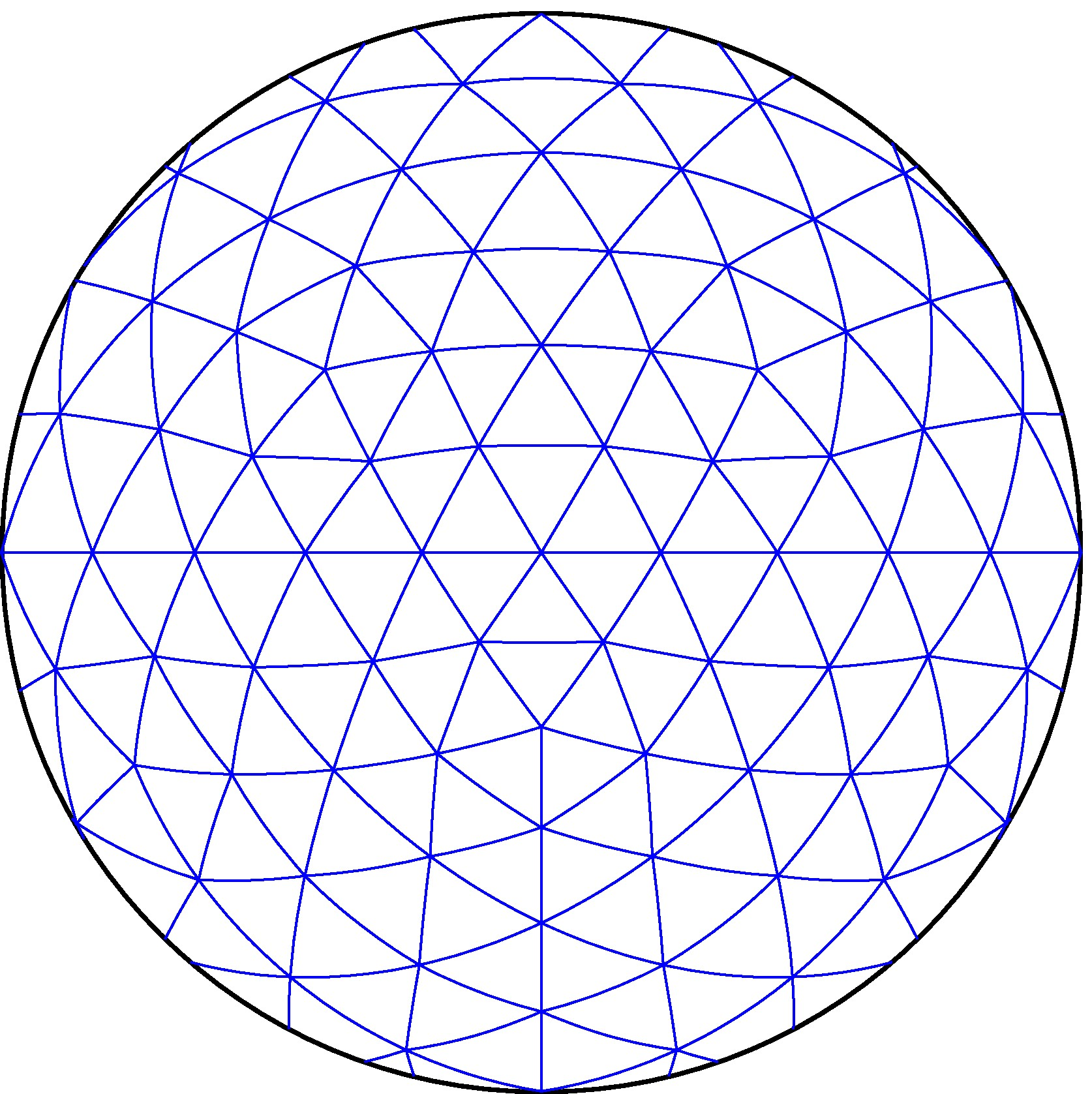}
\caption{\label{fig:refinamenttriangle} Spherical icosahedral grids with levels $0$, $1$  and $2$ with $12,42$  and $162$ vertices, respectively.}
\end{figure}

\subsection{Voronoï-Delaunay decomposition}

Let $\d(\x,\y)$ denote the geodesic distance between $\x$ and $\y$ on $\S$, defined by
\[
\d(\x,\y):=\arccos\langle\x,\y\rangle_{\Rt}\in [0,\pi],
\]
where $\langle\cdot,\cdot\rangle_{\Rt}$ denotes the Euclidean scalar product in $\Rt$. We will use the notation $\a(\cdot)$ and $\l(\cdot)$ for the standard measures of superficial area and curve length, respectively. 
Let $\SN=\{\xip\}_{i=1}^{N}$ denote the set of distinct vertices on $\S$, where $N=2^{2\ell}\cdot10+2$ is the number of vertices and $\ell$ is the level of grid refinement \cite{baumgardner1985icosahedral}. We denote $\Tgsi$ as a geodesic triangle with vertices $\xip,\xjp,\xkp\in\SN$ and define the spherical Delaunay (primal) decomposition on $\S$ as the set $\Ths:=\{\Tgsi:ijk\in\Sigma\}$, where $\Sigma$ is the set of indices such that $i,j,k$  are adjacent neighbors in $\SN$. Here the subscript $h$ denotes the main grid parameter to be defined later.  

The dual Voronoï decomposition of 
$\Ths$ is constructed following (\cite{renka1997algorithm,augenbaum1985construction}). For each vertex $\xip\in\SN$, $1\leq i\leq N$, its associated  Voronoï~cell $\Vi$ is given by
\[
\Vi:=\left\{\x\in\S:\d(\x,\xip)<\d(\x,\xjp),\quad\mbox{for each }1\leq j \leq N,\mbox{ and }j\neq i \right\}.
\]
Each Voronoï cell $\Vi$ consists of all points $\x\in\S$ closer to $\xip$ than any other vertex of $\SN$. Voronoï~cells are open and convex polygons on $\S$, limited by geodesic arcs at their boundaries. In particular, every two cells have an empty intersection, and $\bigcup_{i=1}^{N}\cl(\Vi)=\S$, where $\cl(\cdot)$ denotes the closure of the cell. Further, given two adjacent vertices $\xip$ and $\xjp$, we denote by $\gij=\cl(\Vi)\cap\cl(\Vj)\neq\emptyset$ the geodesic Voronoï~edge on $\S$ associated to the vertices $\xip$ and $\xjp$. Thus, for each vertex $\xip$ we can denote the set of indices of its neighbors $\xjp$ such that $\l(\gij)>0$, \ie, 
\[
\LiN=\left\{j:j\neq i\mbox{ and }\gij=\cl(\Vi)\cap\cl(\Vj)\neq \emptyset\right\}.
\]
Each $\Vi$ has smooth piecewise boundary $\pr\Vi$ formed by the Voronoï~dual edges $\gij$, with $j\in\LiN$, \ie, $\pr\Vi=\bigcup_{j\in\LiN}\gij$. For $\xip$ and $\xjp$ neighboring vertices with $j\in\LiN$, we denote by $\tij$ the Delaunay edge joining vertices $\xip$ and $\xjp$, and by $\xij$ and $\mij$ the midpoints of the geodesic edges $\tij$ (Delaunay) and $\gij$ (Voronoï) respectively. By construction, each geodesic Delaunay edge $\tij$ is perpendicular to geodesic Voronoï edge $\gij$ and the plane formed by $\gij$ and the origin, bisects $\tij$ at its midpoint $\xij$ (see \cite{du2003voronoi,renka1997algorithm}). Therefore $\d(\xip,\x)=\d(\x,\xjp)$, for each $\x\in\gij$, and we denote by $\nxgi$ the co-normal unit vector at the Voronoï edge $\gij$ lying in the plane $\TxS$ tangent to $\S$ at $\x$. Finally, $\nxgi$ is parallel to $\overrightarrow{\xip\xjp}$, \ie, $\nxgi\parallel\overrightarrow{\xip\xjp}$ for each $\x\in\gij$, see Figure \ref{fig:voronoicell}.

\begin{figure}[!h]
\centering
\includegraphics[scale=0.26]{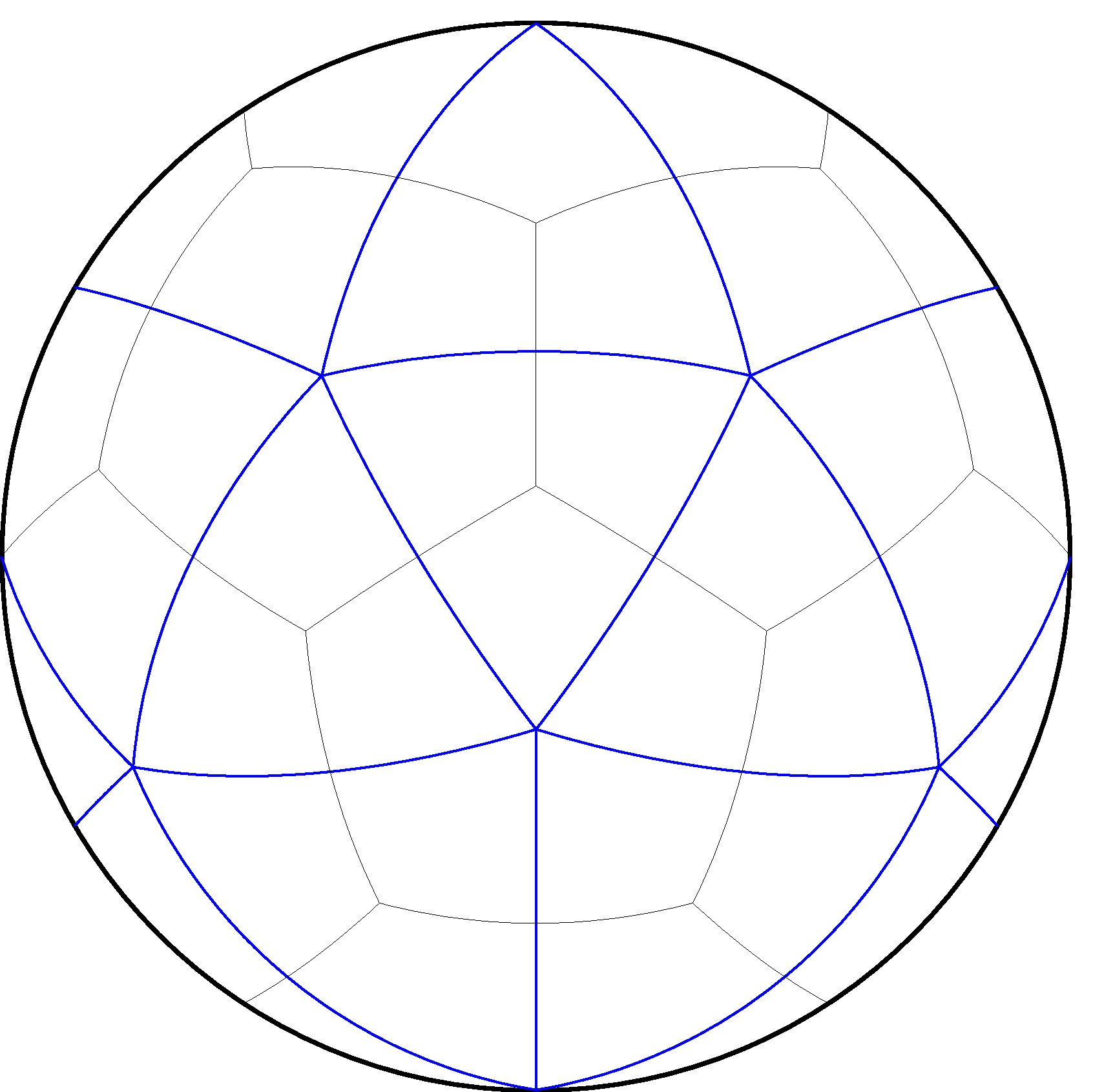}\hspace{0.2cm}
\includegraphics[scale=0.26]{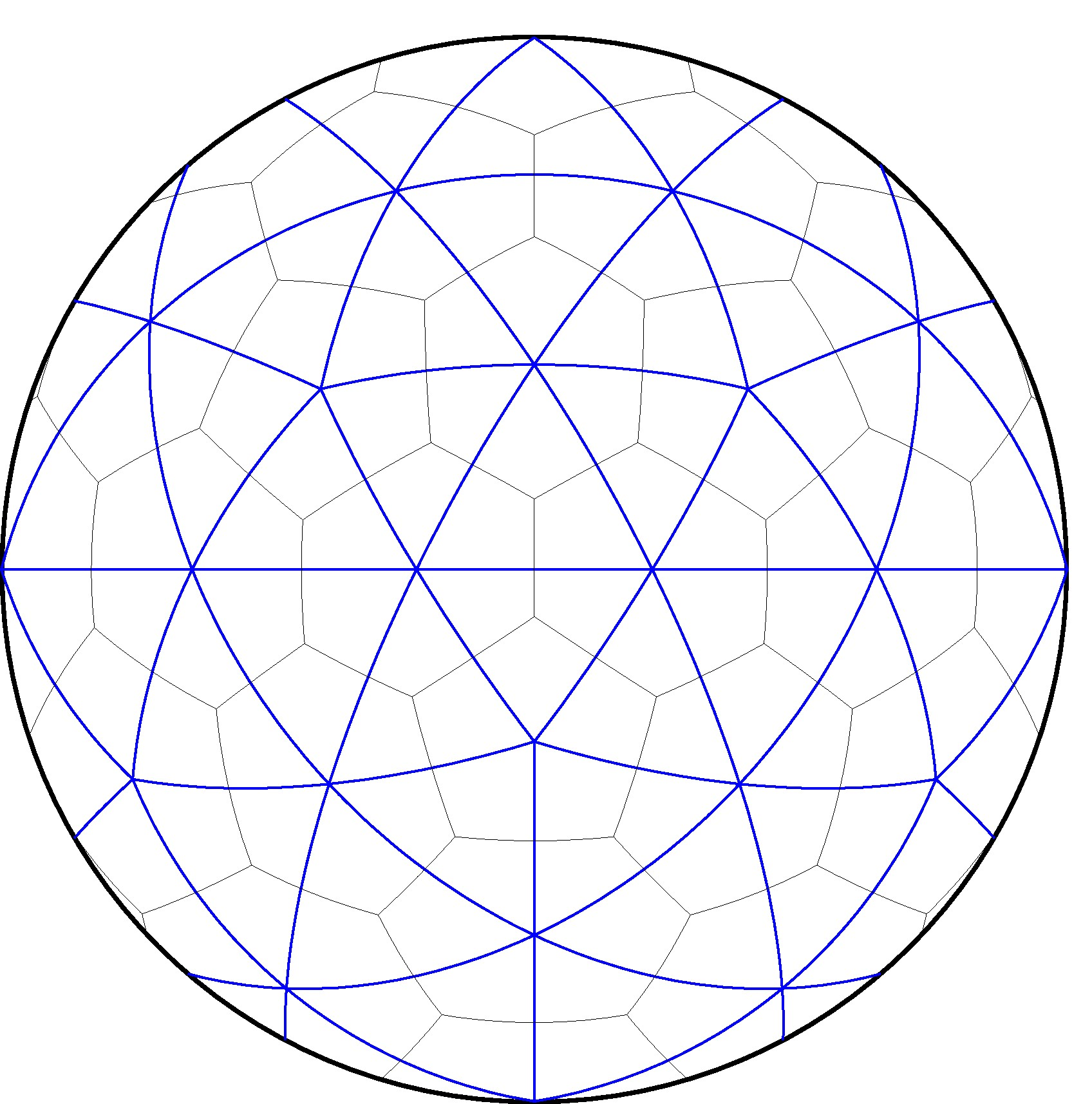}\hspace{0.2cm}
\begin{tikzpicture}[scale=0.6,font=\small]
\coordinate (xi) at (0,0);
\coordinate (xi1) at (3,0);
\coordinate (xi2) at (2.1,2.7);
\coordinate (xi3) at (-1,3);
\coordinate (xi4) at (-3.0,0.8);
\coordinate (xi5) at (-2.1,-2.2);
\coordinate (xi6) at (1.5,-2.6);
%midpoints
\coordinate (mi1) at (1.5,0);
\coordinate (mi2) at (1.1,1.3);
\coordinate (mi3) at (-0.5,1.5);
\coordinate (mi4) at (-1.5,0.4);
\coordinate (mi5) at (-1.0,-1.1);
\coordinate (mi6) at (0.7,-1.3);
\coordinate (m12) at (2.6,1.3);
\coordinate (m23) at (0.6,2.8);
\coordinate (m34) at (-2,1.9);
\coordinate (m45) at (-2.5,-0.7);
\coordinate (m56) at (-0.4,-2.4);
\coordinate (m61) at (2.2,-1.3);
%Voronoi vertices
\coordinate (v1) at (1.5,1);
\coordinate (v2) at (0.5,1.8);
\coordinate (v3) at (-1.3,1.2);
\coordinate (v4) at (-1.7,-0.4);
\coordinate (v5) at (-0.3,-1.8);
\coordinate (v6) at (1.5,-0.8);
\coordinate (ngix) at (1.5,0.3);

\draw[Blue,dashed] (xi) -- (xi1);
\draw[Blue,dashed] (xi)--(xi2);
\draw[Blue,dashed] (xi)--(xi3);
\draw[Blue,dashed] (xi)--(xi4);
\draw[Blue,dashed] (xi)--(xi5);
\draw[Blue,dashed] (xi)--(xi6);
\draw[Blue,dashed] (xi1)--(xi2);
\draw[Blue,dashed] (xi2)--(xi3);
\draw[Blue,dashed] (xi3)--(xi4);
\draw[Blue,dashed] (xi4)--(xi5);
\draw[Blue,dashed] (xi5)--(xi6);
\draw[Blue,dashed] (xi6)--(xi1);
\draw[Black,->] (ngix)--(2.5,0.3) node[right]{$\nxgi$};
\draw (v1)--(v2);
\draw (v2)--(v3);
\draw (v3)--(v4);
\draw (v4)--(v5);
\draw (v5)--(v6);
\draw (v6)--(v1) node[right]{$\gij$};

\filldraw [black,thick](xi1) circle (1.5pt);
\filldraw [black,thick](xi2) circle (1.5pt);
\filldraw [black,thick](xi3) circle (1.5pt);
\filldraw [black,thick](xi4) circle (1.5pt);
\filldraw [black,thick](xi5) circle (1.5pt);
\filldraw [black,thick](xi6) circle (1.5pt);

\filldraw [black,thick](xi) circle (1.5pt) node[below right]{$\xip$};
\filldraw [black,thick](xi1) circle (1.5pt) node[below right]{$\xjp$};
\filldraw [black,thick](mi1) circle(0.5pt) node[below right]{$\xij$};
\filldraw [black,thick](v3) node[below right]{$\Vi$};
\end{tikzpicture}
\caption{\label{fig:voronoicell}Voronoï-Delaunay decomposition on $\S$ with levels $0$ and $1$, and the geometric configuration of Voronoï cell and its associated triangles.}
\end{figure}
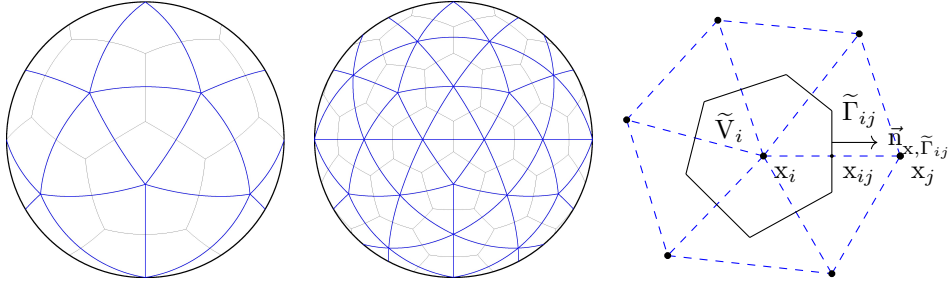
\begin{rem}
The Voronoï-Delaunay decomposition constructed as described above will be called a non-optimized grid and denoted throughout this paper as Voronoï-Delaunay decomposition NOPT.
\end{rem}

\begin{rem} 
In the NOPT decomposition, the constrained cell centroids normally do not coincide with the nodes. However, this can be an important property for discretization. Therefore, we will also consider Spherical Centroidal Voronoï Tessellations (SCVT). These are constructed through an iterative method (Lloyd's algorithm, see \cite{du2003constrained,du2003voronoi}). In the SCVT tessellation, the vertex generating each cell  $\Vi\subset\S$, is the  \emph{constrained cell centroid} $ \xci$,   \ie, the
minimum of the following function
\[
F(\x)=\int_{\Vi}\|\y-\x\|^{2}ds(\y).
\] 
\end{rem}
Given a Voronoï-Delaunay decomposition, we consider some grid parameters previously defined in \cite{du2003voronoi,du2005finite}. Let $\hi=\max_{\x\in \Vi}\d(\xip,\x)$ be the maximum geodesic distance between vertex $\xip$ and the points in its associated cell $\Vi$ and $h=\max_{i=1,\dots,N}h_{i}$.

In addition, we consider the following shape regularity or almost uniform conditions given by \cite{mishev1998finite,li2000generalized,ciarlet2002finite}:
\begin{defn}[Almost uniform]
\label{def:aluniform}
We say that a Voronoï-Delaunay decomposition $\VGs=\{\xip,\Vi\}_{i=1}^{N}$  on $\S$ is finite volume regular if for every spherical polygon $\Vi\in\VGs$ with boundary $\pr\Vi=\bigcup_{j\in\LiN}\gij$, there exist positive constants $C_0$ and $C_1$, independent of $h$ such that
\[
\frac{1}{C_{0}}h\leq \l(\gij) \leq C_{0}h,\quad\mbox{and }\quad
\frac{1}{C_{1}}h^{2}\leq \a(\Vi) \leq C_{1}h^{2}.
\]
\end{defn}
From now onward, we assume that the Voronoï-Delaunay decompositions NOPT and SCVT are almost uniform grids on $\S$.

\subsection{Geometric correspondence}

We describe here geometric relations between $\S$ and its polyhedral approximation $\Sh$, using the framework given by \cite{demlow2009higher,ju2009finite}. First, we assume that $\S$ is approximated by a polyhedral sequence $\Sh$ formed by the decomposition $\Th$ (planar triangles) as $h$ goes to zero. The smooth and bijective mapping $\Prj:\Sh\to \S$ is defined as radial projection of any point $\xp\in\Sh$ onto the spherical surface, \ie, $\Prj(\xp)=\xp/\|\xp\|$ (Figure \ref{fig:radialproj}). Observe that, by construction, the vertices ($\SN=\{\xip\}_{i=1}^{N}$) of the polyhedral $\Sh$ belong to the surface of $\S$ \ie, $\SN=\Sh\cap\S$. This implies that $\Sh=\bigcup_{ijk\in\Sigma}\Tgi$ and $\S=\bigcup_{ijk\in\Sigma}\Tgsi$, where $\Sigma$ is the set of neighboring vertices in $\SN$. Notice that the polyhedral of level zero is the initial icosahedron. 

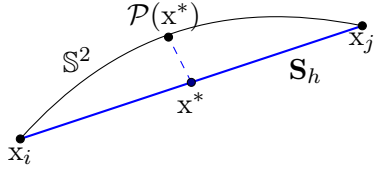
\begin{figure}[th]
\begin{center}
\begin{tikzpicture}[scale=1.5]
\draw [blue,thick](0,0) to (3,1);
\draw (0,0) to[bend left] (3,1);
\draw (0,-0.15) node {$\x_{i}$};
\draw [fill=black](0,0) circle (1.0pt);
\draw (3,0.85) node {$\x_{j}$};
\draw [fill=black](3,1) circle (1.0pt);
\draw [blue,dashed](1.5,0.5) to (1.3,0.9);
\draw (1.26,1.08) node {$\Prj(\xp)$};
\draw [fill=black](1.3,0.9) circle (1.0pt);
\draw (1.5,0.32) node {$\xp$};
\draw [fill=DarkBlue](1.5,0.5) circle (1.0pt);
\draw (2.5,0.6) node {$\Sh$};
\draw (0.5,0.7) node {$\S$};
\end{tikzpicture}
\end{center}
\caption{\label{fig:radialproj}Radial projection of polyhedral $\Sh$ on sphere $\S$.}
\end{figure}

We also consider the following spherical shell in $\Rt$: 
\[
\Oh:=\left\{\xp\in\Rt\setminus\{0\}:1-Ch^{2}<\|\xp\|<1+Ch^{2}\right\},
\]
 with $C$ chosen such that $\S$ and $\Sh$ are contained in $\Oh$.  For functions $u\in\Ht$, we denote by $\Eu$ the extension of $u$  to $\Oh$ given by $\Eu(\xp)=u(\Prj(\xp))$, for each $\xp\in\Sh$. The following result has been shown in  \cite[Proposition 1, pp. 1677]{du2003voronoi}:

\begin{prp}
\label{prp:duextension}
For any $\xp\in\Oh$ and $\x=\Prj(\xp)\in\S$, and $i,j\in\{1,2,3\}$,
\[
\begin{split}
\grads u(\x)=\grad\Eu(\x),& \quad  \grad(\pri\Eu(\x))=\grads(\prsi u(\x))-(\prsi u(\x))\nsx,\\
\|\xp\|\grad\Eu(\xp)=\grad\Eu(\x),& \quad  \|\xp\|^{2}\pri\prj\Eu(\xp)=\pri\prj\Eu(\x),
\end{split}
\]
where $\pri$ denotes partial derivative with respect to $x_i$ and $\prsi$ denotes the $i$-th component of the tangential derivative, for $i=1,2,3$. 
\end{prp}
The following result compares the norms of functions defined on $\S$ and $\Sh$. A proof is given in \cite{demlow2009higher}.

\begin{prp}
\label{prp:equivnormext}
Let $u\in\Wtp$ with $1\leq p\leq \infty$. There exist positive constants $C_0, C_1$ and $C_2$ such that for $h$ small enough,
\[
\begin{split}
\frac{1}{C_{0}}\|u\|_{L^{p}(\Tgsi)}&\leq\|\Eup\|_{L^{p}(\Tgi)}\leq C_{0}\|u\|_{L^{p}(\Tgsi)},\\
\frac{1}{C_{1}}\|\grads u\|_{L^{p}(\Tgsi)}&\leq\|\grad\Eup\|_{L^{p}(\Tgi)}\leq C_{1}\|\grads u\|_{L^{p}(\Tgsi)},\\
\|\grad^{\al}\Eup\|_{L^{p}(\Tgi)}&\leq C_{2}\sum_{0\leq |\al|\leq2}\|\grads^{\al} u\|_{L^{p}(\Tgsi)},
\end{split}
\]
where $\Eup$ is the extension of $u$ to $\Oh$ restricted to $\Sh$, and $\grad^{\al},\grads^{\al}$ denote the usual derivatives and tangential derivatives up to order $2$.
\end{prp}

\section{A Voronoï-based finite volume method}
\label{sec:fvm}

In this section, we seek an approximate solution of \eqref{eq:strongform} via a finite difference/volume scheme. First, from Gauss theorem, we have that
\[
-\int_{\Vi}\laps u(\x)ds(\x)=-\int_{\pr\Vi}\grads u(\x)\cdot\nxvi d\gamma(\x),
\]
where $\nxvi$ denotes the co-normal unit vector on $\pr\Vi$ at $\x$, and $d\gamma(\x)$ is the geodesic length measure. Since $\pr\Vi=\cup_{j\in\LiN}\gij$, we have that
\[
-\sum_{j\in\LiN}\int_{\gij}\grads u(\x)\cdot\nxgi d\gamma(\x)=\int_{\Vi}f(\x)ds(\x).
\]
We denote the continuous flux of $u$ across the edge $\gij$ by
\begin{equation}
\label{eq:cflux}
\Fcij(u):=- \int_{\gij}\grads u(\x)\cdot\nxgi d\gamma(\x),
\end{equation}
and define its central difference approximation
\begin{equation}
\label{eq:dflux}
\Fdij(\uh):=-\l(\gij)\frac{\uh(\xjp)-\uh(\xip)}{\|\xjp-\xip\|}.
\end{equation}
Additionally, for each cell $\Vi$, let $f_{i}$ denote the mean value of the data $f$ on $\Vi$, \ie,
\begin{equation}
\label{eq:fi}
f_{i}=\frac{1}{\a(\Vi)}\int_{\Vi}f(\x)ds(\x).
\end{equation}
The Voronoï-based finite volume scheme is defined as
\begin{equation}
\label{eq:FVscheme}
(\lapd(\uh))_{i}:=\frac{1}{\a(\Vi)}\sum_{j\in\LiN}\Fdij(\uh)=f_{i},\quad\mbox{for }1\leq i\leq N,
\end{equation}
where $\lapd$ is a discretization of the Laplacian, $u_{h}$ is an approximate solution and $f_{i}$ is defined in \eqref{eq:fi}. To emphasize the dependence of the grid, we always will use the subscript $h$. The finite volume scheme \eqref{eq:FVscheme} is conservative, since we have that,
\[
\sum_{i=1}^{N}\sum_{j\in\LiN}\Fdij(\uh)=0,
\]
which follows from $\Fdij=-\Fdji$ if the vertices $\xip$ and $\xjp$ are neighbors with $\l(\gij)>0$.

\subsection{Discrete function spaces}

In this subsection, we introduce some local function spaces on $\S$, as in \cite{dziuk2013finite,demlow2009higher,ju2009finite}. We define the Lagrange finite element space on a polyhedral $\Sh$ as
\[
\Po:=\left\{\uhp\in C^{0}(\Sh):\uhp\big|_{\Tgi},\mbox{is linear affine for each }\Tgi\in\Th\right\},
\]
and the corresponding lifted finite element space on $\S$,
\[
\Pos:=\left\{\uh\in C^{0}(\S):\uh=\uhp\circ\Prji,\,\mbox{for each }\uhp\in\Po\right\},
\]
where $\Prji$ denotes the inverse of the radial projection $\Prj$. For the approximation of functions on $\HqS$, we consider $\U$ the the zero-averaged subspace of $\Pos$ given by
\[
\U(\S):=\left\{\uh\in\Pos:\int_{\S}\uh(\x)ds(\x)=0\right\}.
\]
Notice that, from Proposition \ref{prp:duextension}, we can get $\U\subset\HqS$ endowed with the $H^{1}$-norm. As in \cite{li2000generalized,lin2013finite}, we define the piecewise constant function space associated with the dual Voronoï~decomposition given by
\[
\Vhs(\S):=\left\{v\in\Lt:v\big|_{\Vi},\mbox{ is constant in each }\Vi\mbox{ for }1\leq i\leq N\right\}.
\]
We introduce interpolation operators $\IU$ and $\IV$ mapping functions defined on $\S$ onto $\U$ and $\Vhs$, respectively. Note that, given the function values at the vertices of the  Voronoï grid, the operators are uniquely defined. The following interpolation estimates will be used in our analysis and are shown in \cite[Proposition 3, pp.~1677]{du2005finite}.

\begin{prp}[Interpolation estimates]
\label{prp:interp}
Assume $u\in \WtpHq$ for $2\leq p\leq \infty$  and $v\in \HtHq$. Then, for $h$ small enough, there exist positive constants $C_U$ and $C_V$ independent of $h$ such that,
\[
\begin{split}
\|u-\IU(u)\|_{\Wkp}&\leq C_{U}h^{2-k}\|u\|_{\Wtp},\quad\mbox{ for each }k\in\{0,1\},\\
\|v-\IV(v)\|_{\Lt}&\leq C_{V}h\|v\|_{\Ht}.
\end{split}
\]
\end{prp}

We also define the linear transference interpolation $\I:\U\to\Vhs$ as
\[
\I(\uh)(\x)=\sum_{i=1}^{N}\uh(\xip)\chi_{i}(\x),\quad\mbox{for each }\uh\in\U,
\]
where $\chi_{i}$ represents the characteristic function corresponding to the cell $\Vi$, with $1\leq i\leq N$.

In order to show basic estimates, we shall need the following inverse estimate for finite element functions (see \eg,  \cite{ciarlet2002finite,brenner2007mathematical,demlow2009higher}), which follows from the almost-uniformity of the decomposition $\Ths$. The proof is valid on $\S$ under conditions of Propositions \ref{prp:duextension} and \ref{prp:equivnormext}, cf. \cite[Proposition 2.7, pp.~812]{demlow2009higher} or \cite[Lemma 3.4, pp.~524]{kovacs2018maximum}.

\begin{lem}[Inverse estimate]
\label{lem:propinverse}
Let $\Ths$ be an almost uniform Voronoï-Delaunay decomposition on $\S$. Assume that $l,m$ are non negative integers with $
l\leq m$ and $1\leq p,q\leq \infty$, such that $\U\subset W^{l,p}(\Tgsi)\cap W^{m,q}(\Tgsi)$. Then, there exists a positive constant $C$ independent of $h$, such that $\vh\in\U$ satisfies,
\[
\|\vh\|_{W^{m,p}(\Tgsi)}\leq Ch^{l-m-2\left(1/q-1/p\right)}\|\vh\|_{W^{l,q}(\Tgsi)}.
\]
\end{lem}

We establish the following auxiliary result for later use in the analysis.

\begin{lem}
\label{lem:edgeIntp}
Let $\Tgsi\in\Ths$ be a geodesic triangle of an almost uniform Voronoï-Delaunay decomposition on $\S$ and $\tij\subset\partial\Tgsi$. Then, for $1\leq q<\infty$ and $\vh\in\U$, there exists a positive constant $C$ independent of $h$, such that
\begin{subequations}
\begin{align}
\int_{\tij}[\vh(\x)-\I (\vh)(\x)]d\gamma(\x)&=0,\label{eq:taus}\\
\|\vh-\I(\vh)\|_{\LqT}&\leq Ch\|\vh\|_{\WoqT}.\label{eq:estimIntp}
\end{align}
\end{subequations}
\end{lem}
\begin{proof}
Let $\xij$ be the midpoint of $\tij\subset \partial\Tgsi$, we define $\tij^{(n)}$ as the geodesic between the points $\xij$ and $\xn$ with $n\in\{i,j\}$. We know that $\I(\vh)(\x)=\vh(\xn)$ for each $\x\in\tij^{(n)}$, where $n\in\{i,j\}$ and $\tij=\tij^{(i)}\cup\tij^{(j)}$. We can immediately derive the following estimates,
\begin{align}
\int_{\tij}[\vh(\x)-\I(\vh)(\x)]d\gamma(\x)&=\int_{\tij^{(i)}}[\vh(\x)- \vh(\xip)]d\gamma(\x)+\int_{\tij^{(j)}}[\vh(\x)-\vh(\xjp)]d\gamma(\x)\nonumber\\
&=\int_{\tij}\vh(\x)d\gamma(\x)-\int_{\tij^{(i)}}[\vh(\xip)+\vh(\xjp)]d\gamma(\x)\nonumber\\
&=\int_{\tij}\vh(\x)d\gamma(\x)-\frac{1}{2}\l(\tij)[\vh(\xip)+\vh(\xjp)],\label{eq:lemtri01}
\end{align}
where $\l(\tij)$ denotes the length of $\tij$. Let $\Evh$ be the lift of $\vh$ to the spherical shell $\Oijk:=\left\{\xp\in\Oh:\Prj(\xp)\in\Tgsi\right\}$, with $\vh(\x)=\Evh(\x)$ for all $\x\in\S$. Also let $\xpij$ be the midpoint of $[\x_{i},\x_{j}]$ (the segment goes from $\xip$ to $\xjp$) such that $\xij=\Prj(\xpij)$. Now, assume that $\Evh\in C^{1}(\Oijk)$, by Taylor's Theorem for $\Evh$ around $\xij$ and integrating over $\tij$, we obtain
\begin{align*}
\int_{\tij}\vh(\x)d\gamma(\x)&=\int_{\tij}\Evh(\x)d\gamma(\x)=\l(\tij)\Evh(\xij)\\
&\quad +\int_{\tij}\int_{0}^{1}\nabla \Evh(\xij+t(\x-\xij))\cdot(\x-\xij)dtd\gamma(\x).
\end{align*}
the second term in the right-hand side vanishes by parity of $\nabla\Evh(\y)$ for any $\y\in[\x,\xij]$ and a symmetry of $\tij$ with respect to the midpoint $\xij$. It follows that,
\[
\int_{\tij}\vh(\x)d\gamma(\x)=\l(\tij)\Evh(\xij).
\]
Substituting the expression above into \eqref{eq:lemtri01}, one obtains
\[
\int_{\tij}[\vh(\x)-\I(\vh)(\x)]d\gamma(\x)=\l(\tij)\Evh(\xij)-\frac{1}{2}\l(\tij)[\vh(\xip)+\vh(\xjp)]=0,
\]
which follows from $\Evh(\xij)=\Evh(\xpij)$ and $\Evh(\xpij)=\tfrac{1}{2}[\vh(\x_{i})+v_{h}(\x_{j})]$. This shows the identity \eqref{eq:taus}. For \eqref{eq:estimIntp}, we consider $\Qi,\Qj$ and $\Qk$ the three spherical polygonal regions makeup by the intersection of the triangle $\Tgsi$ with the Voronoï~cells associated to each vertex $\xip,\xjp$ and $\xkp$, \ie, 
\[
\Qn=\Tgsi\cap\Vn,\quad\mbox{for }n\in\{i,j,k\}.
\]
For $\x\in\Qi$ we have $\I(\vh)(\x)=\vh(\xip)$. Then
\begin{align*}
\left|\vh(\x)-\vh(\xip)\right|&=\left|\int_{0}^{1}\nabla \Evh(\xip+t(\x-\xip))\cdot(\x-\xip)dt\right|\\&\leq \|\x-\xip\|\max_{\xp\in[\x,\xip]}\left\|\nabla\Evh(\xp)\right\|,
\end{align*}
where $\Evh$ is the radial extension of $\vh$ to the spherical shell $\Oh$, we used its Taylor expansion and $[\x,\xip]$ is a segment that connects $\x$ with $\xip$. For $1\leq q< \infty$ and, by integrating over $\Qi$, we get
\begin{align*}
\int_{\Qi}\left|\vh(\x)-\I(\vh)(\x)\right|^{q}ds(\x)&\leq \int_{\Qi}\|\x-\xip\|^{q}\max_{\xp\in[\x,\xip]}\left\|\nabla\Evh(\xp)\right\|^{q}ds(\x)\\
&\leq Ch^{q}\int_{\Qi}\max_{\y\in\Qi}\left\|\nabla\Evh(\y)\right\|^{q}ds(\x).
\end{align*}
Recalling Definition \ref{def:aluniform}  and Lemma \ref{lem:propinverse} with $m=l=1$ and $p=\infty$, we can have
\begin{align*}
\|\vh-\I(\vh)\|_{L^{q}(\Qi)}^{q}&\leq Ch^{q}\a(\Qi)\|\nabla \Evh\|_{L^{\infty}(\Qi)}^{q}\leq Ch^{q}h^{2}\|\nabla \Evh\|_{L^{\infty}(\Qi)}^{q}\\
&\leq Ch^{q}h^{2}h^{-2}\|\nabla \Evh\|^{q}_{L^{q}(\Qi)}\leq Ch^{q}\|\nabla \Evh\|^{q}_{L^{q}(\Qi)}.
\end{align*}
The estimates similarly hold for $\Qj$ and $\Qk$. Combining these results, we get that 

\[
\|\vh-\I(\vh)\|_{L^{q}(\Tgsi)}\leq Ch\|\grads\vh\|_{L^{q}(\Tgsi)},
\]
which completes the proof.
\end{proof}

Let us now introduce some discrete norms and seminorms for functions on $\U$. Similarly to \cite{du2003voronoi,droniou2014introduction,bessemoulin2015discrete}, for $1\leq p <\infty$, we denote
\[
\begin{split}
\|\uh\|_{0,p,h}^{p}&=\sum_{i=1}^{N}\a(\Vi)|\uh(\xip)|^{p},\\%
|\uh|_{1,p,h}^{p} & =\sum_{i=1}^{N}\sum_{j\in\LiN}\frac{1}{2}\l(\gij)\d(\xip,\xjp)\left|\frac{\uh(\xip)-\uh(\xjp)}{\|\xip-\xjp\|}\right|^{p},\\
\|\uh\|_{1,p,h}^{p} & =\|\uh\|_{0,p,h}^{p}+|\uh|_{1,p,h}^{p}.
\end{split}
\]
In the case $p=2$, we omit $p$ in our notation and simply write $\|\cdot\|_{0,2,h}=\|\cdot\|_{0,h}$ and $\|\cdot\|_{1,2,h}=\|\cdot\|_{1,h}$ for the norms and $|\cdot|_{1,2,h}=|\cdot|_{1,h}$ for the seminorm. Furthermore, for the case $p=\infty$, we can use the usual notational convention for the $\max$-norm, \ie, $\|\cdot\|_{\Li}$, for functions on $\U$.

\begin{prp}
\label{prp:equivnorm}
For $\uh\in\U$, there exist positive constants $C_{0}$ and $C_{1}$, independent of $h$ such that, 
\begin{subequations}
\begin{align}
\frac{1}{C_{0}}\|\uh\|_{0,p,h}&\leq\|\uh\|_{\Lp}\leq C_{0}\|\uh\|_{0,p,h},\label{eq:mequiv01}\\
\frac{1}{C_{1}}|\uh|_{1,p,h}&\leq \|\grads\uh\|_{\Lp}\leq C_{1}|\uh|_{1,p,h},\label{eq:mequiv02}
\end{align}
\end{subequations}
with $p\in\{1,2\}$.
\end{prp}

\begin{proof}
About $p=2$, we cite \cite[Proposition 4, pp.~1678]{du2005finite} and \cite[Lemma 3.2.1, pp.~124]{li2000generalized}. About $p=1$, from Proposition \ref{prp:equivnormext}, for the extension $\Euhp$ of $\uh$ to $\Oh$ restricted to $\Sh$, we have 
\begin{equation}
\label{eq:equiv01}
\int_{\Tgsi}|\uh(\x)|ds(\x)\leq C\int_{\Tgi}|\Euhp(\xp)|ds(\xp).
\end{equation}

\begin{figure}[!h]
\begin{center}
\begin{tikzpicture}[scale=0.7,
%line cap=round,
%line join=round,
%>=triangle 45,
x=1cm,y=1cm]
\clip(-1.0,-0.5) rectangle (8.0,7.0);
\fill[line width=1.6pt,color=DarkBlue,fill=Blue,fill opacity=0.01] (0,0) -- (4.126324716897398,5.772286864654252) -- (7.02,0.38) -- cycle;
% primal edges
\draw [line width=1.2pt,color=DarkBlue] (0,0)-- (4.126324716897398,5.772286864654252);
\draw [line width=1.2pt,color=DarkBlue] (4.126324716897398,5.772286864654252)-- (7.02,0.38);
\draw [line width=1.2pt,color=DarkBlue] (7.02,0.38)-- (0,0);
% dual edges
\draw [line width=1.2pt,color=BrickRed] (3.416420081044503,1.9187658712304956)-- (2.063162358448699,2.886143432327126);
\draw [line width=1.2pt,color=BrickRed] (3.416420081044503,1.9187658712304956)-- (5.573162358448698,3.0761434323271257);
\draw [line width=1.2pt,color=BrickRed] (3.416420081044503,1.9187658712304956)-- (3.51,0.19);
% triangle regions
\draw [line width=1.2pt,dashed,color=Gray] (0,0)-- (3.416420081044503,1.9187658712304956);
\draw [line width=1.2pt,dashed,color=Gray] (7.02,0.36)--(3.416420081044503,1.9187658712304956);
\draw [line width=1.2pt,dashed,color=Gray] (4.126324716897398,5.772286864654252)--(3.416420081044503,1.9187658712304956);
% Voronoi
\draw [color=BrickRed](3.71,4.0) node[anchor=north west] {$\Qkp$};
\draw [color=BrickRed](1.79,1.2) node[anchor=north west] {$\Qip$};
\draw [color=BrickRed](4.65,1.85) node[anchor=north west] {$\Qjp$};
% inner triangles
\draw [color=Gray](4.5,3.6) node[anchor=north west] {$\Si$};
\draw [color=Gray](1.79,2.7) node[anchor=north west] {$\Sj$};
\draw [color=Gray](3.65,1.2) node[anchor=north west] {$\Sk$};
% Nodes
\draw [fill=black] (0,0) circle (2pt);
\draw[color=black] (-0.24,-0.17) node {$\xip$};
\draw [fill=black] (4.126324716897398,5.772286864654252) circle (2pt);
\draw[color=black] (4.11,6.13) node {$\xkp$};
\draw [fill=black] (7.02,0.36) circle (2pt);
\draw[color=black] (7.4,0.31) node {$\xjp$};
\draw [fill=black] (2.063162358448699,2.886143432327126) ++(-1.5pt,0 pt) -- ++(1.5pt,1.5pt)--++(1.5pt,-1.5pt)--++(-1.5pt,-1.5pt)--++(-1.5pt,1.5pt);
\draw[color=black] (1.78,3.20) node {$\xp_{ik}$};
\draw [fill=black] (5.573162358448698,3.0761434323271257) ++(-1.5pt,0 pt) -- ++(1.5pt,1.5pt)--++(1.5pt,-1.5pt)--++(-1.5pt,-1.5pt)--++(-1.5pt,1.5pt);
\draw[color=black] (5.86,3.34) node {$\xp_{jk}$};
\draw [fill=black] (3.51,0.19) ++(-1.5pt,0 pt) -- ++(1.5pt,1.5pt)--++(1.5pt,-1.5pt)--++(-1.5pt,-1.5pt)--++(-1.5pt,1.5pt);
\draw[color=black] (3.64,-0.16) node {$\xp_{ij}$};
\draw [fill=BrickRed] (3.416420081044503,1.9187658712304956) circle (2pt);
\draw[color=black] (4.2,1.9) node {$\q_{ijk}$};
\end{tikzpicture}
\end{center}
\caption{\label{fig:trianref} Geometric configuration of  planar triangle $\Tgi\in\Th$ with vertices $\xip,\xjp$ and $\xkp$ and its circumcenter denoted by $\q_{ijk}$. }
\end{figure}
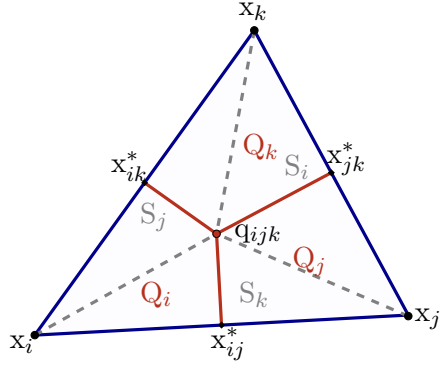

Note that $\Euhp$ is linear in planar triangle $\Tgi\in\Th$ and assume that $\a(\Tgi)=\sum_{n\in\{i,j,k\}}\a(\Sn)$, where $\Sn$ denotes the triangular regions in $\Tgi$, shown enclosed by dashed lines in Figure \ref{fig:trianref}. Then, by numerical integration formula with second-order accuracy, we compute that
%\begin{align*}
%\int_{\Tgi}|\Euhp(\xp)|ds(\xp)=\sum_{n\in\{i,j,k\}}\int_{\Sn}|\Euhp(\xp)|ds(\xp)&=\a(\Sk)|\uhp(\xpij)|+\\&\quad+\a(\Si)|\uhp(\xpjk)|+\a(\Sj)|\uhp(\xpik)|,
%\end{align*}
%\begin{align*}
\[
\sum_{n\in\{i,j,k\}}\int_{\Sn}|\Euhp(\xp)|ds(\xp)=\a(\Sk)|\uhp(\xpij)|+\a(\Si)|\uhp(\xpjk)|+\a(\Sj)|\uhp(\xpik)|,
\]
%\end{align*}
where $\xpij,\xpjk$ and $\xpik$ represent the midpoints of each edge of $\Tgi$. Then, we have
\begin{align}
\sum_{n\in\{i,j,k\}}\int_{\Sn}|\Euhp(\xp)|&ds(\xp)=\a(\Sk)\left|\frac{\uhp(\xip)+\uhp(\xjp)}{2}\right|+\a(\Si)\left|\frac{\uhp(\xjp)+\uhp(\xkp)}{2}\right|\nonumber\\
&\quad+\a(\Sj)\left|\frac{\uhp(\xip)+\uhp(\xkp)}{2}\right|\nonumber\\
&\leq \frac{\a(\Sk)+\a(\Sj)}{2}\left|\uhp(\xip)\right|+\frac{\a(\Sk)+\a(\Si)}{2}\left|\uhp(\xjp)\right|\nonumber\\
&\quad+\frac{\a(\Si)+\a(\Sj)}{2}\left|\uhp(\xkp)\right|\label{eq:equiv02}.
\end{align}

Notice that 
\[
\tfrac{\a(\Sk)+\a(\Sj)}{2}=\a(\Qip),\, \tfrac{\a(\Sk)+\a(\Si)}{2}=\a(\Qjp)\mbox{ and }\tfrac{\a(\Si)+\a(\Sj)}{2}=\a(\Qkp).
\]
In fact, we have $\Qn=\Prj(\Qnp)$ for $n\in\{i,j,k\}$. Thus, gathering \eqref{eq:equiv01} and \eqref{eq:equiv02}, and summing up all triangles of $\Ths$, we obtain
\[
\|\uh\|_{\Lo}=\sum_{\Tgsi\in\Ths}\int_{\Tgsi}|\uh(\x)|ds(\x)\leq C\sum_{i=1}^{N}\a(\Vi)|\uh(\xip)|=C\|\uh\|_{0,1,h},
\]
which yields the right-hand side of \eqref{eq:mequiv01}. Similarly, we can obtain the left-hand side. The inequality \eqref{eq:mequiv02} follows by using the fact that $\uhp$ is linear on each $\Tgi\in\Th$. Furthermore, $\nabla\uhp$ is constant on each $\Tgi\in\Th$; then the result is given by using the numerical integration and central difference approximation with the second order of accuracy. 
\end{proof}

\subsection{A variational formulation}
We now describe a variational formulation for the finite volume scheme. For $(\uh,\vh)\in\U\times\U$, we define the total flux bilinear form $\Ahs:\U\times\Vhs\to\R$ such that
\begin{equation}
\label{eq:biformFV}
\Ahs(\uh,\I(\vh))=\sum_{i=1}^{N}\vh(\xip)\sum_{j\in\LiN}\Fcij(\uh),
\end{equation}
and its discrete version
\begin{equation}
\label{eq:biformFVd}
\Ahd(\uh,\I(\vh))=\sum_{i=1}^{N}\vh(\xip)\sum_{j\in\LiN}\Fdij(\uh),
\end{equation}
where $\Fcij$ and $\Fdij$ are defined in \eqref{eq:cflux} and \eqref{eq:dflux} respectively. So, an approximation $\uh\in\U$ of \eqref{eq:FVscheme} is defined as the unique solution of the discrete problem: find $\uh\in\U$ such that
\begin{equation}
\label{eq:weakformv}
\Ahd(\uh,\I(\vh))=(f,\I(\vh)),\quad\mbox{for each }\vh\in\U,
\end{equation}
In other words, we have, in each Voronoï volume, that $\Ahd(\uh,\chi_{i}) = f_{i}$, where the values of $f_{i}$ are defined in \eqref{eq:fi}.

\begin{prp}[\cite{mishev1998finite,du2003voronoi}]
\label{prop:well-v}
Let $\VGs=\{\xip,\Vi\}_{i=1}^{N}$ be an almost uniform  Voronoï-Delaunay decomposition on $\S$. Consider $\Fdij$ as the discrete flux defined in \eqref{eq:dflux}. Then, for the solution  $\uh\in\U$ of problem \eqref{eq:weakformv}, there are positive constants $C_0$ and $C_1$ such that
\[
\begin{split}
\Ahd(\uh,\I(\vh))&\leq C_{0}|\uh|_{1,h}|\vh|_{1,h},\quad\mbox{for each }\vh\in\U,\\
\Ahd(\uh,\I(\uh))&\geq C_{1}|\uh|_{1,h}^{2}.
\end{split}
\]
Here $|\cdot|_{1,h}$ denotes the discrete seminorm with $p=2$.
\end{prp}
The following result establishes an estimate of the stability of scheme \eqref{eq:FVscheme}, which is an immediate consequence of the proposition above.

\begin{prp}
\label{prp:estabilityv}
Let $f\in \Lt$ satisfying the compatibility condition \eqref{eq:compf}. A unique approximate solution $\uh\in\U$ of the discrete problem \eqref{eq:weakformv} satisfies,
\begin{equation}
\label{eq:stabilityv}
|\uh|_{1,h}\leq C\|f\|_{\Lt},
\end{equation}
where $C$ is a positive constant independent of parameter $h$.
\end{prp}

\subsection{Geometric error estimates}

In this subsection, we present two bounds concerning the geometric perturbation errors in the bilinear forms. We begin with the following lemma.

\begin{lem}
\label{lem:AhsAhd}
Let $\Ahs(\cdot,\cdot)$ and $\Ahd(\cdot,\cdot)$ be the total flux bilinear forms defined in \eqref{eq:biformFV} and \eqref{eq:biformFVd} respectively. Assume that $\uh\in\U$ is the unique solution to the discrete problem \eqref{eq:weakformv}. Then, for each $\vh\in\U$ and $p > 1$ with $1/p+1/q=1$, there exists a positive constant $C$, independent of $h$, such that
\[
\left|\Ahs(\uh,\I(\vh))-\Ahd(\uh,\I(\vh))\right|\leq Ch^{2}\|\grads \uh\|_{\Lp}\|\grads\vh\|_{\Lq}.
\]
\end{lem}
\begin{proof}
This estimate was established for $p=q=2$ in \cite[Lemma 4, pp.~1686]{du2005finite}. The same proof applies to show the estimates for general $p$ and $q$ using a Hölder inequality and norm equivalence from Proposition \ref{prp:equivnorm}.
\end{proof}

We establish the following technical lemma.

\begin{lem}
\label{lem:help}
Let $\x\in\S$ and $\wh\in\U$, then there exists a positive constant $C$, independent of $h$, such that
\begin{equation}
\label{eq:lastlem}
|\laps\wh(\x)|\leq Ch\|\grads \wh\|.
\end{equation}
\end{lem}
\begin{proof}
Take a planar triangle $\Tgi\in\Th$ and its radial projection $\Tgsi=\Prj(\Tgi)$, using $\x=\Prj(\xp) \in \Tgsi$, for $x^*\in \Tgi$. Without loss of generality, to simplify calculations, assume this triangle lies on a plane of constant third coordinate $x_3^*$, so that $\bxat$ of $\xp=(\xaf,\xas,\bxat)\in\Tgi$, with fixed $\bxat$, can fully define the planar triangle in terms of the variables $(\xaf,\xas)$\footnote{Any other planar triangle of the grid can be obtained in this representation by rotations of the sphere.}. 

By construction, we have $1-Ch^{2}<\|\xp\|\leq 1$, for $h$ small enough. Let $\wh:\Rt\setminus\{0\}\to\R$ be linear (affine) on $\Tgi$, that is, $\wh(\xaf,\xas,\bxat)=a\xaf+b\xas+c\bxat+d=a\xaf+b\xas+e$, where $e=c\bxat+d$ is constant and $a,b,c,d,e\in\R$. In spherical coordinates, we get that
\[
\xaf=r\cos\theta\sin\phi,\quad\xas=r\sin\theta\sin\phi,\quad\xat=r\cos\phi,\quad\text{with }\theta\in[-\pi,\pi],\,\phi\in[-\tfrac{\pi}{2},\tfrac{\pi}{2}],
\]
where
\[
r=\|\xp\|,\quad \theta=\arctan\left(\tfrac{\xas}{\xaf}\right),\quad \phi=\arccos\left(\tfrac{\xat}{r}\right).
\]
Then, for $\bxat$ fixed, we can write, 
\[
\wh(\xaf,\xas,\bxat)=a\xaf+b\xas+e=a\bxat\cos\theta\tan\phi+b\bxat\sin\theta\tan\phi+e.
\]
Notice that $\wh(r,\theta,\phi)$ is constant for given fixed $\theta$ and $\phi$, and therefore is invariant with respect to $r$ on $\Tgi$, thus $\wh\in\U$. Now,
\[
\grad\wh(r,\theta,\phi)=\br\pr_{r}\wh(r,\theta,\phi)+\bth\frac{1}{r\sin\phi}\pr_{\theta}\wh(r,\theta,\phi)+\bva\frac{1}{r}\pr_{\phi}\wh(r,\theta,\phi),
\]
where $\{\br,\bth,\bva\}$ are the local orthogonal unit vectors in the directions of increasing $r,\theta$, and $\phi$. Observe that $\pr_{r}\wh(r,\theta,\phi)=0$, so
\[
\grad\wh(r,\theta,\phi)=\bth\frac{1}{r\sin\phi}\tan\phi\left(-a\bxat\sin\theta+b\bxat\cos\theta\right)+\bva\frac{1}{r}\sec^{2}\phi\left(a\bxat\cos\theta+b\bxat\sin\theta\right).
\]
By a similar calculation, we can find that
\begin{align*}
\Delta\wh(r,\theta,\phi)&=\frac{1}{r^{2}}\prt_{\phi}\wh+\frac{1}{\sin^{2}\phi}\prt_{\theta}\wh+\frac{1}{r^{2}}\frac{\cos\phi}{\sin{\phi}}\pr_{\phi}\wh\\
&=\frac{2}{r^{2}}\tan\phi\sec^{2}\phi(a\bxat\cos\theta+b\bxat\sin\theta)-\frac{1}{r^{2}\sin^{2}\phi}\tan\phi(a\bxat\cos\theta+b\bxat\sin\theta)\\
&\quad+\frac{1}{r^{2}}\frac{\cos\phi}{\sin\phi}\sec^{2}\phi(a\bxat\cos\theta\sec^{2}\phi+b\bxat\sin\theta)\\
&=\frac{2}{r^{2}}\tan\phi\sec^{2}\phi(a\bxat\cos\theta+b\bxat\sin\theta)-\frac{1}{r^{2}\sin^{2}\phi}\tan\phi(a\bxat\cos\theta+b\bxat\sin\theta)\\
&\quad+\frac{1}{r^{2}}\frac{1}{\sin^{2}\phi}\frac{\sin\phi}{\cos\phi}(a\bxat\cos\theta+b\bxat\sin\theta)\\
&=\frac{2}{r}\tan\phi(\nabla\wh(r,\theta,\phi)\cdot\bva).
\end{align*}
Then, for $\phi\in[-\tfrac{h}{2},\tfrac{h}{2}]$, and $r$ is in the shell, \ie, $1-Ch^{2}<r<1+Ch^{2}$, we have that
\[
|\Delta\wh(r,\theta,\phi)|\leq \left|\frac{2}{1-Ch^{2}}\tan\tfrac{h}{2}\right|\|\nabla\wh(r,\theta,\phi)\|,
\]
Observe that
%\[
%\frac{1}{1-Ch^{2}}=1+Ch^{2}+C^{2}h^{4}+\cdots,\quad\tan\tfrac{h}{2}=\frac{h}{2}+\frac{1}{3}\left(\frac{h}{2}\right)^{3}+\cdots,
%\]
%then, 
\[
\frac{1}{1-Ch^{2}}\tan\tfrac{h}{2}=\left(1+Ch^{2}+C^{2}h^{4}+\cdots\right)\left(\tfrac{h}{2}+\tfrac{1}{3}\left(\tfrac{h}{2}\right)^{3}+\cdots\right)=h/2+\Or(h^{3}).
\]
Therefore,
\[
|\Delta\wh(r,\theta,\phi)|\leq Ch\|\nabla \wh(r,\theta,\phi)\|.
\]
Taking $r=1$, and using Proposition \ref{prp:duextension}, we arrive at \eqref{eq:lastlem}.
\end{proof}

\begin{lem}
\label{lem:AAhs}
Assume $p>1$ such that $1/p+1/q=1$. Let $\VGs=\{\xip,\Vi\}_{i=1}^{N}$ be an almost uniform Voronoï-Delaunay decomposition on $\S$. Let $\A(\cdot,\cdot)$ and $\Ahs(\cdot,\cdot)$ be the bilinear forms defined in \eqref{eq:biformFE} and \eqref{eq:biformFV} respectively. Also, assume that $\uh\in\U$ is the unique solution of problem  \eqref{eq:weakformv}. Then, there exists a positive constant $C$, independent of $h$, such that
\begin{equation}
\label{eq:lemh2:main}
|\A(\uh,\vh)-\Ahs(\uh,\I(\vh))|\leq C h^{2} \|\grads\uh\|_{\Lp} \|\grads\vh\|_{\Lq},
\end{equation}
for each $\vh\in\U$.
\end{lem}
\begin{proof}
Given $\uh\in\U$, then $\uh\big|_{\Tgsi}\in\WtpT$ for each $\Tgsi\in\Ths$. Multiplying $-\laps\uh$ by $\vh\in\U$ and by using Gauss Theorem, we have
\[
-\int_{\Tgsi}\laps\uh(\x)\vh(\x) ds(\x)=\int_{\Tgsi}\grads\uh(\x)\cdot\grads\vh(\x)ds(\x)-\int_{\pr\Tgsi}\grads\uh(\x)\cdot\nxti\vh(\x)d\gamma(\x).
\]
From definition of $\A(\cdot,\cdot)$ and summing up all $\Tgsi\in\Ths$, we get
\begin{equation}
\label{eq:bifem}
%\begin{split}
\A(\uh,\vh)=\sum_{\Tgsi\in\Ths}\left[\int_{\Tgsi}-\laps\uh(\x)\vh(\x)ds(\x)+\int_{\pr\Tgsi}\grads\uh(\x)\cdot\nxti\vh(\x)d\gamma(\x)\right].
\end{equation}
From Lemma \ref{lem:edgeIntp}, we consider again the spherical polygonal regions $\Qi,\Qj$ and $\Qk$ made of the intersection of the geodesic triangle $\Tgsi$ with the three Voronoï~cells associated to each vertex of the triangle, \ie, $\Qn=\Vn\cap\Tgsi$, $n\in\{i,j,k\}$, with boundaries
\[
\pr\Qn=(\pr\Vn\cap\Tgsi)\cup(\pr\Tgsi\cap\Vn),\quad n\in\{i,j,k\}.
\]
Now, multiplying $-\laps\uh$ by $\I(\vh)\in\Vhs$ and integrating over $\Tgsi$, follows that
\[
\begin{split}
-\int_{\Tgsi}\laps\uh(\x)\I(\vh)(\x)ds(\x)&=\sum_{n=i,j,k}\int_{\Qn}\grads\uh(\x)\cdot\grads\I(\vh)(\x)ds(\x)\\
& \quad-\sum_{n=i,j,k}\int_{\pr\Qn}\grads\uh(\x)\cdot\nxqn \I(\vh)(\x)d\gamma(\x)\\
& = -\sum_{n=i,j,k}\int_{\partial \Qn}\grads u_{h}(\x)\cdot\nxqn \I(\vh)(\x)d\gamma(\x).
\end{split}
\]
Rearranging the boundary $\pr\Qn$ with $n\in\{i,j,k\}$, we have
\[
\begin{split}
-\int_{\Tgsi}\laps\uh(\x)\I(\vh)(\x)ds(\x)& = -\int_{\pr\Tgsi}\grads\uh(\x)\cdot\nxti\I(\vh)(\x) d\gamma(\x)\\
& \quad-\sum_{n=i,j,k}\int_{\pr\Vn\cap\Tgsi}\grads\uh(\x)\cdot\nxvn \I(\vh)(\x)d\gamma(\x).
\end{split}
\]
Now, summing up all geodesic triangles of $\Ths$ and using duality principle, \ie, each edge of $\Tgsi$ intersects to a unique dual Voronoï edge, we get
\begin{align*}
-\sum_{\Tgsi\in\Ths}\int_{\Tgsi}\laps\uh(\x)\I(\vh)(\x)ds(\x) &= -\sum_{\Tgsi\in\Ths}\int_{\pr\Tgsi}\grads\uh(\x)\cdot\nxti\I(\vh)(\x) d\gamma(\x)\\
&\quad -\sum_{i=1}^{N}\int_{\pr\Vi}\grads\uh(\x)\cdot\nxvi \I(\vh)(\x)d\gamma(\x).
\end{align*}
The last term above on the right-hand side is the bilinear form $\Ahs(\cdot,\cdot)$. Therefore
\begin{equation}
\label{eq:bifvm}
\Ahs(\uh,\I(\vh))=\sum_{\Tgsi\in\Ths}\left[-\int_{\Tgsi}\laps\uh(\x)\I(\vh)(\x)ds(\x)+\int_{\pr\Tgsi}\grads\uh(\x)\cdot \nxti\I(\vh)(\x) d\gamma(\x)\right].
\end{equation}
Then, we subtract \eqref{eq:bifvm} from \eqref{eq:bifem} to obtain
\[
\begin{split}
\left|\A(\uh,\vh)-\right.&\left.\Ahs(\uh,\I(\vh))\right|\leq\left|\sum_{\Tgsi\in\Ths}\int_{\Tgsi}\laps\uh(\x)[\vh(\x)-\I(\vh)(\x)]ds(\x)\right|\\
& \quad+\left|\sum_{\Tgsi\in\Ths}\int_{\pr\Tgsi}\grads\uh(\x)\cdot\nxti[\vh(\x)-\I(\vh)(\x)]d\gamma(\x)\right|.
\end{split}
\]
For each $\tij\in\pr\Tgsi$ the function  $\grads\uh(\x)\cdot\nxti[\vh(\x)-\I(\vh)(\x)]$ is anti-symmetric with respect to the edge's midpoint, and therefore as shown in Lemma \ref{lem:edgeIntp} its integral along the edge vanishes. It follows that
\[
\left|\A(\uh,\vh)-\Ahs(\uh,\I(\vh))\right|\leq\left|\sum_{\Tgsi\in\Ths}\int_{\Tgsi}\laps\uh(\x)[\vh(\x)-\I(\vh)(\x)]ds(\x)\right|.
\]
Finally, invoking Lemma \ref{lem:help}, the H\io lder inequality, Proposition \ref{prp:equivnormext} and Lemma \ref{lem:edgeIntp}, we arrive to
\begin{align*}
\left|\A(\uh,\vh)-\Ahs(\uh,\I(\vh))\right|&\leq  \sum_{\Tgsi\in\Ths}\int_{\Tgsi}|\laps\uh(\x)||\vh(\x)-\I(\vh)(\x)|ds(\x)\\
&\leq Ch\sum_{\Tgsi\in\Ths}\int_{\Tgsi}\|\grads\uh\||\vh(\x)-\I(\vh)(\x)|ds(\x)\\
&\leq Ch^{2}\|\grads\uh\|_{\Lp}\|\grads\vh\|_{\Lq}.
\end{align*}
This finishes the proof.
\end{proof}

\begin{lem}
\label{lem:mainfNOPT}
Assume that $f\in L^{p}(\S)$, with $p>1$ satisfies the compatibility condition \eqref{eq:compf}. Then, for each  $\vh\in\U$, with  $1/p+1/q=1$ there exists a positive constant $C$, independent of $h$, such that
\begin{equation}
\label{eq:mainfNOPT}
\left|(f,\vh)-(f,\I(\vh))\right|\leq C h \|f\|_{\Lp}\|\grads\vh\|_{\Lq},
\end{equation}
\end{lem}

\begin{proof}
By applying the H\io lder's inequality with $1/p+1/q=1$ and using Lemma \ref{lem:edgeIntp}, we obtain 
\[
\begin{split}
\left|(f,\vh)-(f,\I(\vh))\right|&=\left|\sum_{\Tgsi\in\Ths}\int_{\Tgsi\in\Ths}f(\x)\left[\vh(\x)-\I(\vh)(\x)\right]ds(\x)\right|\\
& \leq Ch\sum_{\Tgsi\in\Ths}\|f\|_{\LpT}\|\grads\vh\|_{\LqT}\\
&\leq Ch\|f\|_{\Lp}\|\grads\vh\|_{\Lq}.
\end{split}
\]
Therefore, we get \eqref{eq:mainfNOPT}.
\end{proof}

\section{Error analysis}
\label{sec:erroranalysis}
In this section, we establish the estimates of convergence order of the approximate solutions of FVM in classical $H^{1}$, $L^{2}$-norm and $\max$-norm using the framework of \cite{ewing2002accuracy,du2005finite,ju2009finite} for Voronoï-Delaunay decomposition on $\S$ and highlight that the estimated convergence rates depend on the position of vertices of the geometric setting.

\subsection{Classical $H^{1}$ and $L^{2}$ estimates}

The following result provides an error estimate of the finite volume solution $\uh$ in the $H^{1}$ and $L^{2}$-norms with minimal regularity assumptions for the exact solution $u$ and is valid for Voronoï-Delaunay decompositions in general on $\S$. 
\begin{thm}
\label{thm:H1-NOPT}
Let $\VGs=\{\xip,\Vi\}_{i=1}^{N}$ be an almost uniform Voronoï-Delaunay decomposition on $\S$. Assume that $f\in L^{2}(\S)$ satisfies  \eqref{eq:compf} and the unique solution $u$ of \eqref{eq:weakform} belongs to $\Ht\cap\HqS$. Let $\Fdij$ be the discrete flux defined in \eqref{eq:dflux}, such that the discrete problem \eqref{eq:weakformv} has unique solution $\uh\in\U$. Then, there exists a positive constant $C$, independent of $h$, such that 
\begin{subequations}
\begin{align}
\|\grads\eh\|_{\Lt}&\leq Ch \left(\|u\|_{\Ht}+\|f\|_{\Lt}\right),\label{eq:mainH1}\\
\|\eh\|_{\Lt}&\leq Ch \|f\|_{\Lt}+Ch^{2}\|u\|_{\Ht},\label{eq:mainL2}
\end{align}
\end{subequations}
where $\eh=u-\uh$.
\end{thm}

\begin{proof}
Firstly, for \eqref{eq:mainH1} we consider the coercivity of $\A(\cdot,\cdot)$, and define $\vah=\IU(u)-\uh$. Then by  Proposition \ref{prp:interp} and triangular inequality, we get
\begin{equation}
\label{eq:H1M1}
\begin{split}
\|\grads (u-\uh)\|_{\Lt}^{2}&\preceq \left|\A(u-\uh,u-\uh)\right|\\
& \preceq \left|\A(u-\uh,u-\IU(u))\right|+\left|\A(u,\vah)-\A(\uh,\vah)\right|\\
&\preceq \left|\A(u-\uh,u-\IU(u))\right|+\left|\A(u,\vah)-\Ahd(\uh,\I(\vah))\right|\\
& \quad+\left|\Ahd(\uh,\I(\vah))-\Ahs(\uh,\I(\vah))\right|\\
& \quad +\left|\Ahs(\uh,\I(\vah)-\A(\uh,\vah)\right|\\
& = I_{1}+I_{2}+I_{3}+I_{4},
\end{split}
\end{equation}
where the hidden constant in \lq{}\lq{}$\preceq$\rq{}\rq{} comes from the coercivity of $\A(\cdot,\cdot)$. For $I_{1}$, applying the continuity of $\A(\cdot,\cdot)$ and Proposition \ref{prp:interp}, we have 
\begin{align}
I_{1}=|\A(u-\uh,u-\IU(u))|&\leq C\|\grads (u-\uh)\|_{\Lt}\|\grads(u-\IU(u))\|_{\Lt}\nonumber\\
&\leq Ch\|\grads (u-\uh)\|_{\Lt}\|u\|_{\Ht}.\label{eq:H101}
\end{align}
For  $I_{2}$, from right-hand sides of variational problems \eqref{eq:weakform} and \eqref{eq:weakformv} and putting $p=q=2$ in Lemma \ref{lem:mainfNOPT}, we find
\[
I_{2}=\left|\A(u,\vah)-\Ahd(\uh,\I(\vah))\right|=|(f,\vah-\I(\vah))|\leq Ch\|f\|_{\Lt}\|\grads\vah\|_{\Lt}.
\]
%Invoking Proposition \ref{prp:interp}, we have 
%\begin{align*}
%\|\grads\vah\|_{\Lt}&\leq \|\grads(u-\IU(u))\|_{\Lt}+\|\grads(u-\uh)\|_{\Lt}\\
%&\leq Ch\|u\|_{\Ht}+\|\grads(u-\uh)\|_{\Lt}.
%\end{align*}
%Thus, we obtain
%\begin{equation}
%\label{eq:H102}
%I_{2}\leq Ch\|f\|_{\Lt}\|\grads(u-\uh)\|_{\Lt}+Ch^{2}\|f\|_{\Lt}\|u\|_{\Ht}.
%\end{equation}
About $I_{3}$, by Lemma \ref{lem:AhsAhd}, equation \eqref{eq:stabilityv} and  Proposition \ref{prp:equivnorm} follows that
\begin{equation}
\label{eq:H103}
\begin{split}
I_{3}=\left|\Ahd(\uh,\I(\vah))-\Ahs(\uh,\I(\vah))\right|&\leq Ch^{2}\|\grads\uh\|_{\Lt}\|\grads\vah\|_{\Lt}\\
&\leq Ch^{2}\|f\|_{\Lt}\|\grads\vah\|_{\Lt}.
\end{split}
\end{equation}
%From \eqref{eq:stabilityv} and  Proposition \ref{prp:equivnorm} follows that $\|\grads\uh\|_{\Lt}\leq C\|f\|_{\Lt}$. Now, by using a similar argument to $I_{2}$, we have
%\begin{equation}
%\label{eq:H103}
%I_{3}\leq Ch^{2}\|f\|_{\Lt}\|\grads(u-\uh)\|_{\Lt}+Ch^{3}\|f\|_{\Lt}\|u\|_{\Ht}.
%\end{equation}
Similarly, for $I_{4}$ we use Lemma \ref{lem:AAhs} to get
\begin{equation}
\begin{split}
\label{eq:H104}
I_{4}=\left|\Ahs(\uh,\I(\vah))-\A(\uh,\vah)\right|&\leq C h^{2}\|\grads\uh\|_{\Lt}\|\grads\vah\|_{\Lt}\\
&\leq  Ch^{2}\|f\|_{\Lt}\|\grads\vah\|_{\Lt}.
\end{split}
\end{equation}
Gathering \eqref{eq:H1M1} and inequalities \eqref{eq:H101}--\eqref{eq:H104}, we have an expression as
\begin{equation}
\begin{split}
\|\grads(u-\uh)\|_{\Lt}^{2}&\leq C h\|\grads(u-\uh)\|_{\Lt}\|u\|_{\Ht}+Ch\|f\|_{\Lt}\|\grads\vah\|_{\Lt}\\
& \quad+Ch^{2}\|f\|_{\Lt}\|\grads\vah\|_{\Lt}.
\end{split}
\end{equation}
Using Young's and triangle inequalities, and Proposition \ref{prp:interp}, we find
\begin{align*}
\|\grads(u-\uh)\|_{\Lt}^{2}&\preceq \epsilon\left(\|\grads(u-\uh)\|^{2}_{\Lt}+\|\grads\vah\|^{2}_{\Lt}\right)+h^{2}\left(\|u\|^{2}_{\Ht}+\|f\|^{2}_{\Lt}\right)\\
&\quad+h^{4}\|f\|^{2}_{\Lt}\\
&\preceq \epsilon\left(\|\grads(u-\uh)\|^{2}_{\Lt}+Ch^{2}\|u\|^{2}_{\Ht}\right)+h^{2}\left(\|u\|^{2}_{\Ht}+\|f\|^{2}_{\Lt}\right)\\
&\quad+h^{4}\|f\|^{2}_{\Lt}.
\end{align*}
Taking $\epsilon$ small enough, we arrive to
\[
\|\grads(u-\uh)\|_{\Lt}^{2}\preceq h^{2}\left(\|u\|^{2}_{\Ht}+\|f\|^{2}_{\Lt}\right)+h^{4}\|f\|^{2}_{\Lt}.
\]
Finally, for $h$ small enough, we obtain \eqref{eq:mainH1}. 

In order to prove \eqref{eq:mainL2}, we derive an error estimate following classical duality argument: for $u-\uh\in \HqS$, from \eqref{eq:weakform}, there exists a unique weak solution $w\in \HtHq$ satisfying
\[
\A(v,w)=(v,u-\uh),\quad \mbox{for all }v\in\HqS.
\]
Taking $v=u-\uh$ and by using the regularity estimate \eqref{eq:regularity}, we have
\begin{equation}
\label{eq:regularityw}
\|w\|_{\Ht}\leq C\|u-\uh\|_{\Lt}.
\end{equation}
Now, assume that $\wh=\IU(w)\in\U$, we obtain 
\[
\begin{split}
\|u-\uh\|_{\Lt}^{2}&=(u-\uh,u-\uh)\leq  |\A(u-\uh,w)|\\
&=|\A(u-\uh,w-\wh)+\A(u,\wh)-\A(\uh,\wh)|\\
&\leq |\A(u-\uh,w-\wh)|+|\A(u,\wh)-\Ahd(\uh,\I(\wh))|\\
& \quad+|\Ahd(\uh,\I(\wh))-\Ahs(\uh,\I(\wh))|\\
&\quad+|\Ahs(\uh,\I(\wh))-\A(\uh,\wh)|\\
&= I_{1}+I_{2}+I_{3}+I_{4}.
\end{split}
\]
About $I_{1}$, by equation \eqref{eq:mainH1}, Proposition \ref{prp:interp} and inequality \eqref{eq:regularityw}, we have 
\begin{equation}
\label{eq:01L2NOPT}
\begin{split}
I_{1}=|\A(u-\uh,w-\wh)|&\leq C \|\grads(u-\uh)\|_{\Lt}\|\grads(w-\wh)\|_{\Lt}\\
&\leq C h^{2}\left(\|u\|_{\Ht}+\|f\|_{\Lt}\right)\|u-\uh\|_{\Lt}.
\end{split}
\end{equation}
For $I_{2}$, by Lemma \ref{lem:mainfNOPT} and expression \eqref{eq:regularityw},  we have
\begin{equation}
\label{eq:02L2NOPT}
I_{2}=|\A(u,\wh)-\Ahd(\uh,\I(\wh))|\leq Ch\|f\|_{\Lt}\|u-\uh\|_{\Lt}.
\end{equation}
Analogously for $I_{3}$, using Lemma \ref{lem:AhsAhd}, equations \eqref{eq:stabilityv} and \eqref{eq:regularityw} follows that
\begin{equation}
\label{eq:03L2NOPT}
I_{3}=\left|\Ahd(\uh,\I(\wh))-\Ahs(\uh,\I(\wh))\right|\leq C h^{2}\|f\|_{\Lt}\|u-\uh\|_{\Lt}.
\end{equation}
Finally, by Lemma \ref{lem:AAhs} along with estimates \eqref{eq:stabilityv} and \eqref{eq:regularityw} yields the expression
\begin{equation}
\label{eq:04L2NOPT}
I_{4}=\left|\Ahs(\uh,\I(\wh))-\A(\uh,\wh)\right|\leq C h^{2}\|f\|_{\Lt}\|u-\uh\|_{\Lt}.
\end{equation}
Combining \eqref{eq:01L2NOPT}--\eqref{eq:04L2NOPT} with $h>0$ small enough, we have
\[
\|u-\uh\|_{\Lt}^{2}\preceq (h \|f\|_{\Lt}+h^{2}\|f\|_{\Lt}+h^{2}\|u\|_{\Ht})\|u-\uh\|_{\Lt}.
\]
Dividing by $\|u-\uh\|_{\Lt}$ encloses the proof of the theorem.
\end{proof}

\begin{rem}
Notice that the convergence rate in the $L^{2}$-norm is lower than that obtained in FEM. As we will see in the numerical experiment, the errors of the finite volume solutions behave better than the estimates given above. Meanwhile, no known standard method exists to increase the convergence estimate in the $L^{2}$-norm by using Voronoï-Delaunay decomposition in general. 
\end{rem}

The dominant term in the bounds of Theorem \ref{thm:H1-NOPT} is given by Lemma \ref{lem:mainfNOPT}. One way to gain extra power in the convergence rate is to use the SCVT optimizations in the standard decomposition. A quadratic order estimate was reported by \cite{du2005finite}. Still, we will illustrate below that a decrease in regularity in the exact solution is possible using the techniques investigated by \cite{ewing2002accuracy}.

The following lemma is a simplified version of a result given by \cite[Lemma 1, pp.~1682]{du2005finite}, assuming that the density function $\rho$ is equal to $1$.
 
\begin{lem}[\cite{du2005finite}]
\label{lem:h2du}
Let $\VGs=\{\xip,\Vi\}_{i=1}^{N}$ be an almost uniform   Spherical Centroidal Voronoï-Delaunay decomposition (SCVT) on $\S$. Then, for any $w\in \HtHq$, there exists a positive constant $C$, independent of $h$, such that
\[
\left|\int_{\Vi}[w(\x)-\IV(w)(\x)]ds(\x)\right|\leq Ch^{2}\a(\Vi)^{1/2}\|w\|_{H^{2}(\Vi)}.
\]
\end{lem}
In light of the lemma above, we have the following result.
\begin{lem}
\label{lem:fh2SCVT}
Let $f\in\HqS$ be a function that satisfies the compatibility condition \eqref{eq:compf} and $\VGs=\{\xip,\Vi\}_{i=1}^{N}$ be an almost uniform Spherical Centroidal Voronoï-Delaunay decomposition (SCVT) on $\S$. Then, for each $w\in \HtHq$, there exists a positive constant $C$, independent of $h$, such that  
\[
\left(f,\IU(w)-\IV(w)\right)\leq Ch^{2}\|f\|_{\Ho}\|w\|_{\Ht},
\]
where $\IU$ and $\IV$ are the interpolation operator on spaces $\U$ and $\Vhs$ respectively.
\end{lem}

\begin{proof}
By definition of the inner product we have
%\[
%\begin{split}
%\left(f,\IU(w)-\IV(w)\right)&=\int_{\S}f(\x)\left[\IU(w)(\x)-\IV(w)(\x)\right]ds(\x)\\
%&=\sum_{\Tgsi\in\Ths}\int_{\Tgsi}f(\x)\left[\IU(w)(\x)-w(\x)\right]ds(\x)\\
%&\quad+\sum_{i=1}^{N}\int_{\Vi}f(\x)\left[w(\x)-\IV(w)(\x)\right]ds(\x).
%\end{split}
%\]
\[
\begin{split}
\left(f,\IU(w)-\IV(w)\right)&=\sum_{\Tgsi\in\Ths}\int_{\Tgsi}f(\x)\left[\IU(w)(\x)-w(\x)\right]ds(\x)\\
&\quad+\sum_{i=1}^{N}\int_{\Vi}\left[f(\x)-\PV(f)\right]\left[w(\x)-\IV(w)(\x)\right]ds(\x)\\
& \quad+\sum_{i=1}^{N}\int_{\Vi}\PV(f)\left[w(\x)-\IV(w)(\x)\right]ds(\x)\\
& = I_{1}+I_{2}+I_{3},
\end{split}
\]
where $\PV(f)$ denotes the $L^{2}$-projection of the function $f$ on $\Vhs$. For $I_{1}$, by using Cauchy-Schwarz inequality and Proposition \ref{prp:interp}, we obtain
\begin{equation}
\label{eq:L2h2dua1}
\begin{split}
I_{1}&\leq \sum_{\Tgsi\in\Ths}\left(\int_{\Tgsi}|f(\x)|^{2}ds(\x)\right)^{1/2}\left(\int_{\Tgsi}|\IU(w)(\x)-w(\x)|^{2}ds(\x)\right)^{1/2}\\
&\leq Ch^{2}\|f\|_{\Lt}\|w\|_{\Ht}.
\end{split}
\end{equation}
Analogously, for $I_{2}$, 
\begin{equation}
\label{eq:L2h2dua2}
\begin{split}
I_{2}&\leq \sum_{i=1}^{N}\left(\int_{\Vi}|f(\x)-\PV(f)(\x)|^{2}ds(\x)\right)^{1/2}\left(\int_{\Vi}|w(\x)-\IV(w)(\x)|^{2}ds(\x)\right)^{1/2}\\
&\leq Ch\sum_{i=1}^{N}\|\grads f\|_{L^{2}(\Vi)}\|w-\IV(w)\|_{L^{2}(\Vi)}\leq Ch^{2}\| f\|_{\Ho}\|w\|_{\Ht}.
\end{split}
\end{equation}
Finally,  for $I_{3}$, using Lemma \ref{lem:h2du} and Proposition \ref{prp:equivnorm}, we have that
\begin{equation}
\label{eq:L2h2dua3}
\begin{split}
I_{3}&\leq Ch^{2}\sum_{i=1}^{N}|\PV(f)\big|_{\Vi}\a(\Vi)^{1/2}\|w\|_{H^{2}(\Vi)}\\
&\leq Ch^{2}\left(\sum_{i=1}^{N}\a(\Vi)\left|\PV(f)\big|_{\Vi}\right|^{2}\right)^{1/2}\left(\sum_{i=1}^{N}\|w\|_{H^{2}(\Vi)}\right)^{1/2}\\
&\leq Ch^{2}\|f\|_{\Lt}\|w\|_{\Ht}.
\end{split}
\end{equation}
Combining \eqref{eq:L2h2dua1}--\eqref{eq:L2h2dua3} we obtain the result.
\end{proof}

We will show below a subtle modification of the proof given in \cite{du2005finite}, assuming that the exact solution $u\in \HtHq$ and source term $f\in\HqS$.

\begin{thm}
\label{thm:L2SCVT}
Let $\VGs=\{\xip,\Vi\}_{i=1}^{N}$ be an almost uniform  Spherical Centroidal Voronoï-Delaunay decomposition (SCVT) on $\S$. Assume that $f\in\HqS$ and that $u\in \HtHq$ is the unique solution of  \eqref{eq:strongform}. Let $\Fdij$ be the discrete flux defined in \eqref{eq:dflux}, such that the discrete problem \eqref{eq:weakformv} has a unique solution $\uh\in\U$. Then, there exists a positive constant $C$, independent of $h$, such that
\[
\|\eh\|_{\Lt}\leq C h^{2}\left(\|u\|_{\Ht}+\|f\|_{\Ho}\right),
\]
where $\eh=u-\uh$.
\end{thm}
\begin{proof}
The proof is analogous to \eqref{eq:mainL2} in Theorem \ref{thm:H1-NOPT}, now  using Lemma \ref{lem:fh2SCVT} into equation \eqref{eq:02L2NOPT} to guarantee an extra power in the term $I_{2}$, \ie, 
\[
I_{2}=\left|\left(f,\wh-\I(\wh)\right)\right|\leq Ch^{2}\|f\|_{\Ho}\|w\|_{\Ht}\leq Ch^{2}\|f\|_{\Ho}\|u-\uh\|_{\Lt},
\]
where $\wh=\IU(w)$. Therefore, the quadratic order is shown and completes the proof of the theorem.
\end{proof}

\subsection{Pointwise error estimates}

In this subsection, we find error estimates in the maximum norm for problem \eqref{eq:strongform}. We shall use the variational formulation \eqref{eq:biformFE} for regularized Green's functions. We consider some properties of these auxiliary functions along with punctual estimates of FEM on surfaces defined by \cite{demlow2009higher,kroner2017approximative}.

\subsubsection{Regularity properties of Green's functions}
The lemma below has fundamental properties of the regularized Green's functions on $\S$. The proof is detailed in \cite[Theorem 4.13, pp.~108]{aubin1998some}.

\begin{lem}[\cite{aubin1998some}]
There exists $G(\x,\y)$, a Green's function for the Laplacian on $\S$ satisfying, for each $u\in \Ct$ and $\x,\y\in\S$ with $\x\neq\y$,
\[
u(\y)=\frac{1}{\a(\S)}\int_{\S}u(\x)ds(\x)-\int_{\S}G(\x,\y)\laps u(\x)ds(\x),
\]
and there exist positive constants $C_{0},C_{1}$ and $C_{2}$ such that
\begin{align*}
|G(\x,\y)|&\leq C_{0}(1+|\ln\d(\x,\y)|),\\
|\grads G(\x,\y)|&\leq C_{1}\frac{1}{|\d(\x,\y)|},\quad |\laps G(\x,\y)|\leq C_{2}\frac{1}{|\d(\x,\y)|^{2}},
\end{align*}
where $\grads$ and $\laps$ denote the tangential gradient and Laplacian acting on a function of $\x$. Finally, $G(\x,\y)$ satisfies
\begin{equation}
\label{eq:compaticondG}
\int_{\S}G(\x,\y)ds(\x)=0.
\end{equation}
\end{lem}
 
From now on, we denote by $\Gyx$ the Green's function $G(\x,\y)$. Further, we consider the notation $\Gyd$ to refer to the discrete Green's function defined as
\begin{equation}
\label{eq:disgreenf}
\Gydx:=\int_{\S}\Gxz\dyz ds(\z),
\end{equation}
where $\dyz$ represents the discrete Dirac delta function at the point $\y$. 

We now present a proposition by Demlow  \cite[Proposition 2.8, pp.~813]{demlow2009higher}, as it will be used further in the analysis.

\begin{prp}[\cite{demlow2009higher}]
\label{prp:demlow-delta}

Consider $\vh\in\U$ and fix $\y\in\Tgsi\subset\S$. Let $\nyti$ be a unit vector on the tangent plane $\TyS$ at $\y$. Then, there exist $\dy$ and $\Edy$, both independent of $\vh$, such that for some positive constant $C$,
\[
\|\dy\|_{\WmpT}+\|\Edy\|_{\WmpT}\leq Ch^{-m-2\left(\frac{p-1}{p}\right)},
\]
for $m\in\{0,1\}$ and $1\leq p\leq \infty$. Further, there exists positive generic constants $C_{0}$  and $C_{1}$, such that 
\begin{align*}
|\vh(\y)|&\leq C_{0}\left|\int_{\Tgsi}\dyx\vh(\x)ds(\x)\right|,\\
\left|\grads \vh(\y)\cdot \nyti\right|&\leq C_{1}\left|\int_{\Tgsi}\vh(\x)\grads\cdot\Edy (\x) ds(\x)\right|,
\end{align*}
where $\Edy=\|\xp\|\Edyp\nyti$, with $\y=\Prj(\yp)$. 
\end{prp}

Let us now introduce a variational formulation for the discrete Green's functions. Fix $\y\in\S$, then we consider two kinds of regularized Green's functions $\Gydo,\Gyda\in C^{\infty}(\S)$ satisfying the following variational problems:
\begin{subequations}
\begin{align}
\A(v,\Gydo)&=(v,\dy),\quad\mbox{for each }v \in\HqS,\label{eq:weakformGd0}\\
\A(v,\Gyda)&=(v,\grads \cdot \Edy),\quad\mbox{for each }v \in\HqS.\label{eq:weakformGd1}
\end{align}
\end{subequations}
Accordingly, we define $\Gydoh,\Gydah\in\U$ as the solutions for the finite element approximate problems 
\begin{equation}
\label{eq:weakformGh0}
\A(\wh,\Gydoh)=(\wh,\dy),\quad\mbox{for each }\wh\in\U,
\end{equation}
and
\begin{equation}
\label{eq:weakformGh1}
\A(\wh,\Gydah)=(\wh,\grads \cdot \Edy),\quad\mbox{for each }\wh\in\U.
\end{equation}
The finite element approximation $\Gydnh\in\U$ with $n\in\{0,1\}$ is taken to be the unique solution of problem
\begin{equation}
\label{eq:weakformGh}
\A(\wh,\Gydn-\Gydnh)=0,\quad\mbox{for each }\wh\in\U.
\end{equation}
Notice that $\Gydnh$ satisfies \eqref{eq:compaticondG} from the structure of the space $\U$ for $n\in\{0,1\}$. On the other hand, the term $\Gydn-\Gydnh$ satisfies the error estimates in $H^{1}$ and $L^{2}$-norm \cite{demlow2009higher}. These discrete Green's functions appear because we will need some additional \emph{a priori} estimates, which are well studied in the literature by \cite{demlow2009higher,kroner2017approximative,rannacher1982some,schatz1995interior}. 

\begin{lem}%[\cite{demlow2009higher}]
\label{lem:limitGreen}
Let $\Gydn$ be a discrete Green's function and $\Gydnh\in\U$ its finite element approximation with $n\in\{0,1\}$. Then, we have:
\begin{subequations}
\begin{align}
\|\grads(\Gyda-\Gydah)\|_{L^{1}(\S)}&\leq C,\label{eq:G1}\\
\|\grads \Gyda\|_{L^{1}(\S)}+\|\Gydo\|_{W^{2,1}(\S)}&\leq C|\ln h|,\label{eq:G2}\\
\|\grads\Gydoh\|_{\Lt}&\leq C |\ln h|^{1/2}.\label{eq:G3}
\end{align}
\end{subequations}
where $C$ are positive generic constants independent of $h$. Here the factor $|\ln h|^{1/2}$ has order  $\Or(h^{-\eta})$ for $\eta\in(0,1)$.
\end{lem}

\begin{proof}
We omit the details of \eqref{eq:G1} and \eqref{eq:G2}, but we highlight that these proofs are detailed in \cite[Lemma 3.3, pp.~819]{demlow2009higher} and \cite[Lemma 5.2, pp.~10]{kroner2017approximative}. About \eqref{eq:G3}, let $\Gydoh\in\U$ be the finite element approximation of $\Gydo$. From equation   \eqref{eq:weakformGh0} follows that
\[
\|\grads \Gydoh\|^{2}_{\Lt}=\A(\Gydoh,\Gydoh)=(\Gydoh,\dy).
\]
From discrete Sobolev inequality (see \eg, \cite[Lemma 4.9.2, pp.~124]{brenner2007mathematical} or \cite[Lemma 3.12, pp.~527]{kovacs2018maximum}), there exists a positive constant $C$, such that
\begin{equation}
\label{eq:dSo2}
\left|\Gydoh(\y)\right|\leq C|\ln h|^{1/2}\|\grads\Gydoh\|_{\Lt}.
\end{equation}
Therefore, we obtain \eqref{eq:G3}.
\end{proof}

Now, we can show a weak stability condition of approximate solutions of \eqref{eq:FVscheme} in the $\max$-norm. 
\begin{prp}
\label{lem:weaklnhv2}
Assume that $\uh\in\U$ is the unique solution of \eqref{eq:weakformv}. Then, there exists a positive  constant $C$, independent of $h$, such that 
\[
\|\uh\|_{\Li}\leq C|\ln h|^{1/2}\|f\|_{\Lt}.
\]
\end{prp}

\begin{proof}
From Proposition \ref{prp:demlow-delta}, there exists $\dy$ supported in $\Tgsi\in\Ths$ such that $\left|\uh(\y)\right|=\left|(\uh,\dy)\right|=\left|\A(\uh,\Gydoh)\right|$. In virtue of continuity of $\A(\cdot,\cdot)$ and Lemma \ref{lem:limitGreen}, we have 
\[
\left|\uh(\y)\right|=\left|\A\left(\uh,\Gydoh\right)\right|\leq C\|\grads\uh\|_{\Lt}\|\grads\Gydoh\|_{\Lt}\leq C|\ln h|^{1/2}\|\grads\uh\|_{\Lt}.
\]
Finally, we get the desired inequality by applying Propositions \ref{prp:estabilityv} and \ref{prp:equivnorm}.
\end{proof}

We now state and show the main results of this section. 

\begin{thm}
\label{thm:maxH1-NOPT}
Let $\VGs=\{\xip,\Vi\}_{i=1}^{N}$ be an almost uniform Voronoï-Delaunay decomposition on $\S$. Assume that $f\in \Lt$ satisfies \eqref{eq:compf} and the unique weak solution $u$ of \eqref{eq:strongform} belongs to $\Wti\cap\HtHq$. Let $\Fdij$ be the discrete flux defined in \eqref{eq:dflux}, such that the discrete problem \eqref{eq:weakformv} has a unique solution $\uh\in\U$. Then, there exists a positive constant $C$, independent of $h$, such that
\[
\|\eh\|_{\Li}\leq C h |\ln h|^{1/2}\left(\|u\|_{\Ht}+\|f\|_{\Lt}\right)+C h^{2}\|u\|_{\Wti},
\]
where $\eh=u-\uh$.
\end{thm}

\begin{proof}
We shall fix $\y\in\S$ and consider a discrete Green's function $\Gydo$ satisfying \eqref{eq:compaticondG} and the problem \eqref{eq:weakformGd0}. We also consider the finite element approximation $\Gydoh\in\U$ satisfying the problem  \eqref{eq:weakformGh}. By Proposition \ref{prp:demlow-delta}, there exists $\dy$ supported in $\Tgsi\in\Ths$, we then have 
\begin{equation}
\label{eq:MmaxH1}
\begin{split}
\left|(u-\uh)(\y)\right|&\leq \left|(u-\IU(u))(\y)\right|+\left|(\IU(u)-\uh)(\y)\right|\\
&\preceq\left|(u-\IU(u))(\y)\right|+\left|\int_{\S}(\IU(u)-\uh)(\x)\delta^{\y}(\x)ds(\x)\right|\\
&\preceq\|u-\IU(u)\|_{\Li}+\left|\A(\IU(u)-\uh,\Gydo)\right|\\
&\preceq \|u-\IU(u)\|_{\Li}+\left|\A\left(\IU(u)-\uh,\Gydo-\Gydoh\right)\right|\\
&\quad+\left|\A\left(\IU(u)-\uh,\Gydoh\right)\right|
\\
&=I_{1}+I_{2}+I_{3}.
\end{split}
\end{equation}
About $I_{1}$, using Proposition \ref{prp:interp} with $p=\infty$, we  have
\begin{equation}
\label{eq:maxHI1}
I_{1}=\|u-\IU(u)\|_{\Li}\leq C h^{2}\|u\|_{\Wti}.
\end{equation}
For $I_{2}$, from \eqref{eq:weakformGh}, we obtain
\begin{equation}
\label{eq:maxHI2}
I_{2}=0.
\end{equation}
Finally, for $I_{3}$, utilizing the continuity of $\A(\cdot,\cdot)$ and Proposition \ref{prp:interp}, Theorem \ref{thm:H1-NOPT} and inequality \eqref{eq:G3} from Lemma  \ref{lem:limitGreen}, we have
\begin{equation}
\label{eq:maxHI3}
\begin{split}
I_{3}&\leq C\|\grads(\IU(u)-u_{h})\|_{\Lt}\|\grads\Gydoh\|_{\Lt}\\
&\leq C\left(\|\grads(u-\IU(u))\|_{\Lt}+\|\grads(u-\uh)\|_{\Lt}\right)\|\grads\Gydoh\|_{\Lt}\\
& \leq Ch|\ln h|^{1/2}\left(\|u\|_{\Ht}+\|f\|_{\Lt}\right).
\end{split}
\end{equation}
Combining \eqref{eq:maxHI1}-\eqref{eq:maxHI3} and \eqref{eq:MmaxH1} for $h>0$ small enough, we find
\[
\left|(u-\uh)(\y)\right|\leq Ch|\ln h|^{1/2}\left(\|u\|_{\Ht}+\|f\|_{\Lt}\right)+Ch^{2}\|u\|_{\Wti}.
\]
Finally, taking the maximum value leads to the desired result.
\end{proof}

Notice that the finite volume scheme \eqref{eq:FVscheme} is included in the estimate of $H^{1}$-norm. Further, the result above is not optimal concerning the regularity required of the exact solution \cite{ewing2002accuracy}. This excessive regularity can be removed as follows: put the restriction in the weak solution $u$ of \eqref{eq:strongform}, belonging to $\WtiHq$. Next, estimate a $max$-norm for the tangential gradient of the solutions by using Propositions \ref{prp:interp}, \ref{prp:demlow-delta}, and Lemma \ref{lem:limitGreen}. Finally, compute the error estimates of approximate solutions in the $\max$-norm. 

In order to prove that, the following result provides a pointwise error estimate for the tangential gradient of the approximate solution. 

\begin{thm}
\label{thm:W1-NOPT}
Let $\VGs=\{\xip,\Vi\}_{i=1}^{N}$ be an almost uniform Voronoï-Delaunay  decomposition on $\S$. Assume that $f\in\Li$ satisfies \eqref{eq:compf} and that the unique weak solution $u$ of \eqref{eq:strongform} belongs to $\WtiHq$. Let $\Fdij$ be a discrete flux \eqref{eq:dflux}, such that the discrete problem \eqref{eq:weakformv} has a unique solution $\uh\in\U$. Then, there exist positive constants $C$ and $h_{0}$ independent of $u$, such that for $0<h\leq h_{0}<1$,
\[
\|\grads\eh\|_{\Li}\leq C h |\ln h|\left(\|u\|_{\Wti}+\|f\|_{\Li}\right),
\]
where $\eh=u-\uh$.
\end{thm}

\begin{proof}
We will proceed via a duality argument: we fix $\y\in\S$ and consider $\nyti\in\TyS$ a tangent unit vector, and from Proposition \ref{prp:demlow-delta}, there exists $\Edy$ supported in $\Tgsi$. Let $\Gyda$ be a discrete Green's function satisfying \eqref{eq:compaticondG} and the variational problem \eqref{eq:weakformGd1}. Also, we consider the finite element approximation $\Gydah$ as the unique solution of \eqref{eq:weakformGh}. From the triangular inequality, we have
\begin{equation}
\label{eq:MW1}
\begin{split}
\left|\grads(u-\uh)(\y)\cdot\nyti\right|&\leq \left|\grads(u-\IU(u))(\y)\cdot\nyti\right|\\
&\quad+\left|\grads(\IU(u)-\uh)(\y)\cdot\nyti\right|\\
&\preceq\|\grads(u-\IU(u))\|_{\Li}\\
&\quad+\left|\int_{\S}(\IU(u)-\uh)(\x)\grads\cdot\Edy(\x) ds(\x)\right|\\
&\preceq \|\grads(u-\IU(u))\|_{\Li}+\left|\A(\IU(u)-\uh,\Gyda)\right|\\
&\preceq \|\grads(u-\IU(u))\|_{\Li}+\left|\A(\IU(u)-\uh,\Gyda-\Gydah)\right|\\
&\quad +\left|\A(\IU(u)-u,\Gydah)\right|+\left|\A(u-\uh,\Gydah)\right|\\
& = I_{1}+I_{2}+I_{3}+I_{4}.
\end{split}
\end{equation}

About $I_{1}$, applying Proposition \ref{prp:interp}, we have
\begin{equation}
\label{eq:W1I1}
I_{1}=\|\grads(u-\IU(u))\|_{\Li}\leq Ch\|u\|_{\Wti}.
\end{equation}

For $I_{2}$, from \eqref{eq:weakformGh}, we obtain 
\begin{equation}
\label{eq:W1I2}
I_{2}=\left|\A(\IU(u)-\uh,\Gyda-\Gydah)\right|=0.
\end{equation}

For $I_{3}$, using the continuity of $\A(\cdot,\cdot)$, Proposition \ref{prp:interp}, the triangular inequality and equations \eqref{eq:G2} and \eqref{eq:G1} from Lemma \ref{lem:limitGreen}, we have
\begin{equation}
\label{eq:W1I3}
\begin{split}
I_{3}&\leq C\|\grads(u-\IU(u))\|_{\Li}\|\grads\Gydah\|_{L^{1}(\S)}\\
&\leq C\|\grads(u-\IU(u))\|_{\Li}\left(\|\grads(\Gyda-\Gydah)\|_{L^{1}(\S)}+\|\grads\Gyda\|_{L^{1}(\S)}\right)\\
&\leq Ch(1+|\ln h|)\|u\|_{\Wti}.
%\leq Ch|\ln h|\|u\|_{\Wti}
\end{split}
\end{equation}

Now about $I_{4}$, by using the linearity of $\A(\cdot,\cdot)$ and adding up and subtracting the total fluxes $\Ahd(\cdot,\cdot)$ and $\Ahs(\cdot,\cdot)$, we obtain
\begin{equation}
\label{eq:W1I4}
\begin{split}
I_{4}&\leq \left|\A(u,\Gydah)-\Ahd(\uh,\I(\Gydah))\right|+\left|\Ahd(\uh,\I(\Gydah))-\Ahs(\uh,\I(\Gydah))\right|\\
&\quad +\left|\Ahs(\uh,\I(\Gydah))-\A(\uh,\Gydah)\right|\\
&= I_{4,1}+I_{4,2}+I_{4,3}.
\end{split}
\end{equation}
For $I_{4,1}$, applying Lemmas \ref{lem:mainfNOPT} and \ref{lem:limitGreen} for $f\in\Li$ and $\Gydah\in\U$, we find that
\begin{equation}
\label{eq:W1I41}
\begin{split}
I_{4,1}&=\left|(f,\Gydah-\I(\Gydah))\right| \leq Ch\|f\|_{\Li}\|\grads\Gydah\|_{\Lo}\\
&\leq Ch(1+|\ln h|)\|f\|_{\Li}.
%&\leq Ch|\ln h|\|f\|_{\Li}.
\end{split}
\end{equation}
For $I_{4,2}$, by using Lemma \ref{lem:AhsAhd} and collecting the inequalities \eqref{eq:G1} and \eqref{eq:G2} from Lemma \ref{lem:limitGreen}, we arrive at
\begin{equation}
\label{eq:W1I42}
\begin{split}
%=\left|\Ahd(\uh,\I(\Gydah))-\Ahs(\uh,\I(\Gydah))\right|
I_{4,2}\leq Ch^{2}\|\grads \uh\|_{\Li}\|\grads\Gydah\|_{\Lo}\leq Ch^{2}(1+|\ln h|)\|\grads \uh\|_{\Li}.
\end{split}
\end{equation}
Similarly, for $I_{4,3}$, from Lemma \ref{lem:AAhs} and Lemma \ref{lem:limitGreen}, we obtain
\begin{equation}
\label{eq:W1I43}
%\begin{split}
%=|\Ahs(\uh,\I(\Gydah))-\A(\uh,\Gydah)|&
I_{4,3}\leq Ch^{2}\|\grads\uh\|_{\Li}\|\grads\Gydah\|_{\Lo}\leq Ch^{2}(1+|\ln h|)\|\grads\uh\|_{\Li}.
% \end{split}
\end{equation}

Notice that from triangle inequality, we obtain
\begin{equation}
\label{eq:W1Iprp}
\|\grads \uh\|_{\Li}\leq \|\grads(u-\uh)\|_{\Li}+\|\grads u\|_{\Li}.
\end{equation}
Thus, gathering all the estimates \eqref{eq:W1I41}-\eqref{eq:W1Iprp} into \eqref{eq:W1I4}, follows that
\begin{equation}
\label{eq:W1I4m}
I_{4}\leq Ch(1+|\ln h|)\|f\|_{\Li}+Ch^{2}(1+|\ln h|)\left(\|\grads(u-\uh)\|_{\Li}+\|\grads u\|_{\Li}\right).
\end{equation}
Finally, combining \eqref{eq:W1I4m} with \eqref{eq:MW1}--\eqref{eq:W1I3}, for $h_{0}\in\Rp$, such that $0<h\leq h_{0}<1$ and by applying the maximum value, we find
\[
\|\grads(u-\uh)\|_{\Li}\leq Ch|\ln h|\left(\|u\|_{\Wti}+\|f\|_{\Li}\right),
\]
which leads to the desired result.
\end{proof}

Now, we show error estimates for approximate solution in $\max$-norm.

\begin{thm}
\label{thm:max-NOPT}
Under the assumptions of Theorem \ref{thm:W1-NOPT}. Then, there exist positive constants $C$ and $h_{0}$ independent of $u$, such that for $0<h\leq h_{0}<1$,
\[
\|\eh\|_{\Li}\leq C h |\ln h|\left(\|u\|_{\Wti}+\|f\|_{\Li}\right),
\]
where $\eh=u-\uh$.
\end{thm}

\begin{proof}
We proceed similarly to Theorem \ref{thm:W1-NOPT}. We fix $\y\in\S$, and from Proposition \ref{prp:demlow-delta}, there exists a smooth function $\dy$ supported in $\Tgsi\in\Ths$. Let $\Gydo$ be a discrete Green's function satisfying \eqref{eq:compaticondG} and the variational problem \eqref{eq:weakformGd0}. We also consider the finite element approximation $\Gydoh$ as a unique solution of the problem \eqref{eq:weakformGh}. By using triangular inequality, we have
\begin{equation}
\label{eq:maxM}
\begin{split}
\left|(u-\uh)(\y)\right|&\leq \left|(u-\IU(u))(\y)\right|+\left|(\IU(u)-\uh)(\y)\right|\\
& \preceq \|u-\IU(u)\|_{\Li}+\left|\int_{\S}(\IU(u)-\uh)(\x)\dy(\x)ds(\x)\right|\\
& \preceq \|u-\IU(u)\|_{\Li}+\left|\A(\IU(u)-\uh,\Gydo)\right|\\
& \preceq \|u-\IU(u)\|_{\Li}+\left|\A(\IU(u)-\uh,\Gydo-\Gydoh)\right|\\
& \quad +\left|\A(\IU(u)-u,\Gydoh)\right|+\left|\A(u-\uh,\Gydoh)\right|\\
& = I_{1}+I_{2}+I_{3}+I_{4}.
\end{split}
\end{equation}
For $I_{1}$, from Proposition \ref{prp:interp}, we have
\begin{equation}
\label{eq:maxI1}
I_{1}=\|u-\IU(u)\|_{\Li}\leq C h^{2}\|u\|_{\Wti}.
\end{equation}
For $I_{2}$, from \eqref{eq:weakformGh}, we have
\begin{equation}
\label{eq:maxI2}
I_{2}=\A(\IU(u)-\uh,\Gydo-\Gydoh)=0. 
\end{equation}
For $I_{3}$, using the continuity of $\A(\cdot,\cdot)$, Proposition \eqref{prp:interp} and Lemma \ref{lem:limitGreen} with \eqref{eq:G3}, we obtain
\begin{equation}
\label{eq:maxI3}
I_{3}=\left|\A(\IU(u)-u,\Gydoh)\right|\leq Ch|\ln h|^{1/2}\|u\|_{\Wti}.
\end{equation}
About $I_{4}$, analogously to Theorem \ref{thm:W1-NOPT}, by using the linearity of $\A(\cdot,\cdot)$ and the total fluxes $\Ahs(\cdot,\cdot)$ and $\Ahd(\cdot,\cdot)$, we have
\begin{equation}
\label{eq:maxI4}
\begin{split}
%=\left|\A(u-\uh,\Gydoh)\right|
I_{4}&\leq \left|\A(u,\Gydoh)-\Ahd(\uh,\I(\Gydoh))\right|+\left|\Ahd(\uh,\I(\Gydoh))-\Ahs(\uh,\I(\Gydoh))\right|\\
& \quad+\left|\Ahs(\uh,\I(\Gydoh))-\A(\uh,\Gydoh)\right|\\
& = I_{4,1}+I_{4,2}+I_{4,3}.
\end{split}
\end{equation}
For $I_{4,1}$, by using Lemmas \ref{lem:mainfNOPT} and \ref{lem:limitGreen}, we find
\begin{equation}
\label{eq:maxI41}
I_{4,1}=\left|\A(u,\Gydoh)-\Ahd(\uh,\I(\Gydoh))\right|=\left|(f,\Gydoh-\I(\Gydoh))\right|\leq Ch|\ln h|\|f\|_{\Li}.
\end{equation}
For $I_{4,2}$, Lemma \ref{lem:AhsAhd} yields
\[
I_{4,2}=\left|\Ahd(\uh,\I(\Gydoh))-\Ahs(\uh,\I(\Gydoh))\right|\leq Ch^{2}|\ln h|\|\grads \uh\|_{\Li}.
\]
Applying Theorem \ref{thm:W1-NOPT}, there exists $h_{0}\in\Rp$ such that for all $0<h\leq h_{0}$, we have
\[
\begin{split}
\|\grads \uh\|_{\Li}\leq Ch|\ln h|\left(\|u\|_{\Wti}+\|f\|_{\Li}\right)+\|\grads u\|_{\Li}.
\end{split}
\]
Then
\begin{equation}
\label{eq:maxI42}
I_{4,2}\leq Ch^{2}|\ln h|\|u\|_{\Wti}.
\end{equation}
As for $I_{4,3}$, using Lemma \ref{lem:AhsAhd}, we have
\begin{equation}
\label{eq:maxI43}
I_{4,3}\leq Ch^{2}|\ln h|\|u\|_{\Wti}.
\end{equation}
Consequently, for $h>0$ small enough and gathering \eqref{eq:maxI4} with \eqref{eq:maxI41}--\eqref{eq:maxI43}, we have
\begin{equation}
\label{eq:maxI4m}
I_{4}\leq Ch|\ln h|\|f\|_{\Li}.
\end{equation}
Finally, combining \eqref{eq:maxM} with \eqref{eq:maxI1}--\eqref{eq:maxI3} and \eqref{eq:maxI4m}, we find
\[
|(u-\uh)(\y)|\leq C h|\ln h|\left(\|u\|_{\Wti}+\|f\|_{\Li}\right).
\]
Therefore, applying the maximum value yields the expected result.
\end{proof}

To end this section, we can get an additional estimate using a Voronoï-Delaunay decomposition SCVT.

\begin{thm}
Let $\VGs=\{\xip,\Vi\}_{i=1}^{N}$ be an almost uniform  Spherical Centroidal Voronoï-Delaunay decomposition (SCVT) on $\S$. Assume that $f\in\HqS$ and that $u\in \Wti\cap\HtHq$ is the unique solution of \eqref{eq:strongform}. Let $\Fdij$ be the discrete flux defined in \eqref{eq:dflux}, such that the discrete problem \eqref{eq:weakformv} has a unique solution $\uh\in\U$. Then, there exist positive constants $C$ and $h_{0}$, independent of $u$, such that for $0<h\leq h_{0}<1$,
\[
\|\eh\|_{\Li}\leq Ch^{2}\|u\|_{\Wti}+ Ch\left(\|u\|_{\Ht}+\|f\|_{\Ho}\right),
\]
where $\eh=u-\uh$.
\end{thm}

\begin{proof}
From triangular inequality, Proposition \ref{prp:interp}, Lemma \ref{lem:propinverse} and Theorem \ref{thm:L2SCVT}, we then obtain
\begin{align*}
\|u-\uh\|_{\Li}&\leq \|u-\IU(u)\|_{\Li}+\|\IU(u)-\uh\|_{\Li}\\
&\leq Ch^{2}\|u\|_{\Wti}+Ch^{-1}\|\IU(u)-\uh\|_{\Lt}\\
&\leq Ch^{2}\|u\|_{\Wti}+Ch\left(\|u\|_{\Ht}+\|f\|_{\Ho}\right).
\end{align*}
For $h$ small enough, the proof is complete.
\end{proof}

\section{Numerical example and final remarks}
\label{sec:numericalexps}

This section illustrates an example of the FV approach of the Laplace-Beltrami operator using the recursive Voronoï-Delaunay decomposition. We consider three types of grids: the non-optimized grid (NOPT) and two of the most used grid optimizations in the literature \cite{miura2005comparison},  the Heikes and Randall optimized grids (HR$95$), proposed in  \cite{heikes1995anumerical,heikes1995bnumerical}, and the \emph{Spherical Centroidal Voronoï~Tessellations} (SCVT) described by Du and collaborators in \cite{du2003constrained,du2003voronoi}. We consider the SCVT grids with constant density function ($\rho=1$). To verify the error estimate $\varepsilon_{h}$, we shall use the example defined in  \cite{heikes1995bnumerical}. The exact solution $u$ in geographic coordinates $(\phi,\theta)$ is defined as
\begin{equation}
\label{eq:ex1solu}
u(\phi, \theta)=\cos\theta\cos^{4}\phi,
\end{equation}
and the forcing source as,
\begin{equation}
\label{eq:ex1solf}
\begin{split}
f(\phi,\theta)= -\cos\theta\cos^{2}\theta[\cos^{2}\phi-4\sin\phi\cos\phi\sin\phi\cos\phi\\
-12\cos^{2}\phi+16\cos^{2}\phi\cos^{2}\phi]/\cos^{2}\phi,
\end{split}
\end{equation}
where $\phi\in[-\pi/2,\pi/2]$ is the latitude and $\theta\in[-\pi,\pi]$ is the longitude.

\begin{figure}[!h]
\centering
\subfloat[NOPT]{
\centering
\includegraphics[scale=0.38]{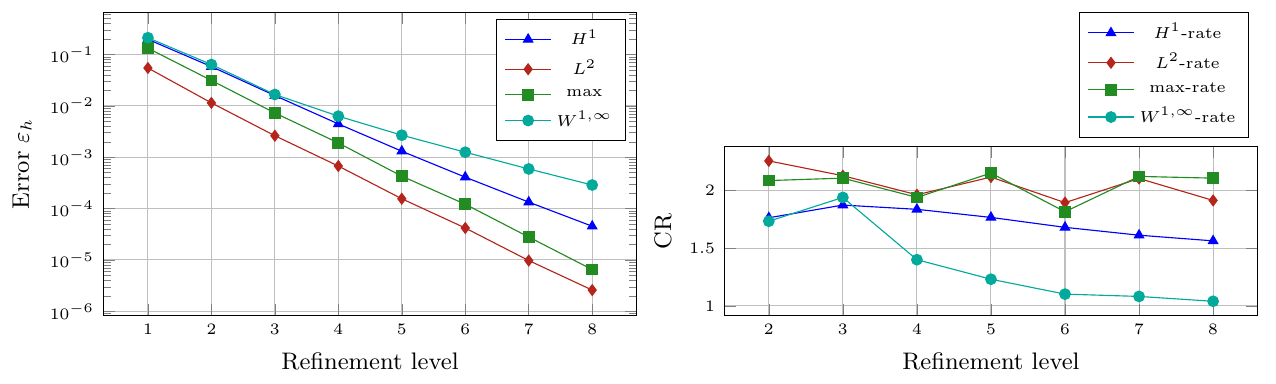}
}
\\
\subfloat[SCVT]{
\centering
\includegraphics[scale=0.38]{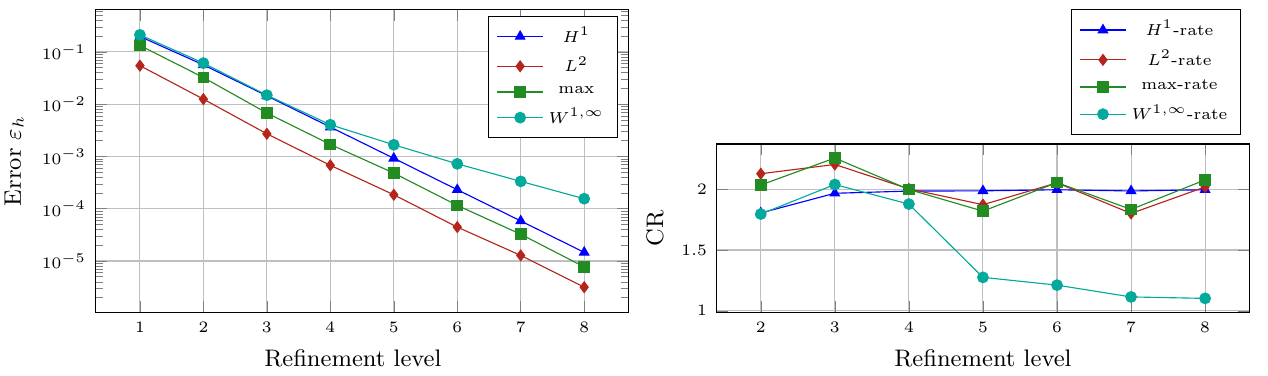}
}
\\
\subfloat[HR$95$]{
\centering
\includegraphics[scale=0.38]{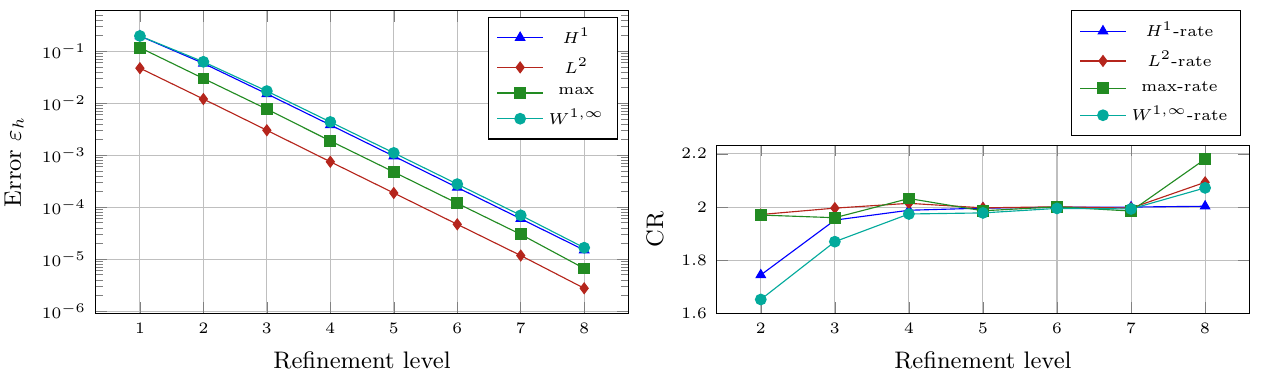}
}
\caption{\label{fig:solnorm} Errors $\eh$ and numerical convergence rates $CR$ in the finite volume approximation for problem \eqref{eq:ex1solf} with exact solution \eqref{eq:ex1solu} in $H^{1}$, $L^{2}$, $\max$ and $W^{1,\infty}$-norm using grids NOPT, SCVT and HR$95$ with different refinement levels.}
\end{figure}

Approximate solutions were obtained using the finite volume scheme \eqref{eq:FVscheme}. Figure \ref{fig:solnorm} shows error estimates and convergence rates of the approximate solution of the problem \eqref{eq:ex1solf} in $H^{1}$, $L^{2}$, $W^{1,\infty}$ and $\max$-norm using three types of grids and different refinement levels. The numerical convergence rate $CR$ with respect to the norm $\|\cdot\|_{\star}$ is given as
\[
CR=\frac{\left|\ln\|\varepsilon_{n}\|_{\star}-\ln\|\varepsilon_{n-1}\|_{\star}\right|}{\ln 2},\quad\mbox{for }n=2,\dots,N,
\] 
where $\varepsilon_{n}=u-\uh$ denotes the error of the $n$-th level. 

Firstly, by using grid NOPT, we observe that the numerical convergence rate is just about $\Or(h)$ in $H^{1}$-norm and matches with the theoretical convergence rate predicted in Theorem \ref{thm:H1-NOPT}. Analogously, from Theorem \ref{thm:W1-NOPT}, we have that the theoretical prediction for the solution error is $\Or(h|\ln h|)$ in $W^{1,\infty}$-norm. At the same time, the obtained logarithmic factor is not detected numerically. 

Furthermore, the numerical convergence rates for problem \eqref{eq:ex1solf} indicate that the error given in $L^{2}$ and $\max$-norm tends to be quadratic order as we refine the grid NOPT. In general, the difference between the right-hand sides of the variational problems \eqref{eq:weakform} and \eqref{eq:weakformv}, when using non-optimized Voronoï-Delaunay decomposition, is a dominant factor for optimal convergence rate (just about $\Or(h)$), as predicted in Theorems \ref{thm:H1-NOPT} and \ref{thm:max-NOPT} respectively. To our knowledge, no existing analytical results confirm the quadratic order in these norms for general Voronoï-Delaunay tessellations. 

However, observe that the numerical convergence rate for SCVT is just about $\Or(h^{2})$ in the $L^{2}$-norm, as had been shown by Du in \cite{du2005finite}.
Here, we modify the proof by using a minimal regularity requirement in the exact solution and highlight that the quadratic order error estimates depend on the geometric criterion of the SCVT. The numerical convergence rates in $\max$-norm is also $\Or(h^{2})$. This latter case is under study, and results will be presented elsewhere. Note also that the best behavior of error estimates is given in $H^{1}$-norm in both optimized grids. Thus, there exists a degree of superconvergence of the approach on gradients. To date, there seem to be no existing theoretical criteria that prove these behaviors and consequently, this brings a good challenge for future research. Extensions of the analysis to HR$95$ grids are our current investigation.

\section*{Declaration of competing interest}

Each author has contributed substantially to conducting the underlying research and writing this manuscript. Furthermore, they have no financial or other conflicts of interest to disclose.

\section*{Acknowledgements}
This work is in memory of Saulo R. M. Barros, who participated in this work but tragically passed away in July 2021, prior to its conclusion. The work presented here was supported by FAPESP (Fundação de Amparo à Pesquisa do Estado de São Paulo) through grant 2016/18445-7 and 2021/06176-0 and also by Coordenação de Aperfeiçoamento de Pessoal de Nível Superior - Brasil (CAPES), Finance Code 001.

\bibliographystyle{plain}
\bibliography{references}
\end{document}